\documentclass[english,reqno]{amsart}

\usepackage{latexsym}
\usepackage{amssymb,amsbsy,amsmath,amsfonts,amssymb,amscd,amsrefs}
\usepackage{dsfont}
\usepackage{mathrsfs}
\usepackage[nolabel]{showlabels} % change outer to nolabel to turn off
\usepackage{enumitem}

\usepackage{float}
\usepackage[caption = false]{subfig}
\usepackage[final]{graphicx}

\usepackage{subfig}
\usepackage{xcolor}

\usepackage{cases}
\theoremstyle{plain}

\newtheorem{theorem}{Theorem}
 
\newtheorem{lemma}{Lemma}
\newtheorem{proposition}{Proposition}
\newtheorem{corollary}{Corollary}

\theoremstyle{definition} % For roman text in the body
\newtheorem{definition}{Definition}
\newtheorem{remark}{Remark}

\newcommand{\vep}{\varepsilon}
\newcommand{\rd}{\mathrm{d}}

\newcommand\defeq{\mathrel{\stackrel{\makebox[0pt]{\mbox{\normalfont\tiny def}}}{=}}}
\newcommand{\dsim}{\stackrel{d}{\sim}}

\newcommand{\Amat}{\mathsf{A}}
\newcommand{\Bmat}{\mathsf{B}}

\newcommand{\amat}{\mathsf{a}}
\newcommand{\xmat}{\mathsf{x}}
\newcommand{\bmat}{\mathsf{b}}
\newcommand{\Smat}{\mathsf{S}}
\newcommand{\Rmat}{\mathsf{R}}

\newcommand{\Fmat}{\mathsf{F}}
\newcommand{\Gmat}{\mathsf{G}}
\newcommand{\Xmat}{\mathsf{X}}
\newcommand{\fmat}{\mathsf{f}}
\newcommand{\gmat}{\mathsf{g}}
\newcommand{\pmat}{\mathsf{p}}
\newcommand{\qmat}{\mathsf{q}}
\newcommand{\Pmat}{\mathsf{P}}
\newcommand{\Qmat}{\mathsf{Q}}
\newcommand{\ymat}{\mathsf{y}}

\newcommand{\Asf}{\mathsf{A}}
\newcommand{\Bsf}{\mathsf{B}}
\newcommand{\xsf}{\mathsf{x}}
\newcommand{\bsf}{\mathsf{b}}
\newcommand{\Ssf}{\mathsf{S}}

\newcommand{\Idsf}{\mathsf{Id}}
\newcommand{\fsf}{\mathsf{f}}
\newcommand{\gsf}{\mathsf{g}}
\newcommand{\psf}{\mathsf{p}}
\newcommand{\qsf}{\mathsf{q}}
\newcommand{\Psf}{\mathsf{P}}
\newcommand{\Qsf}{\mathsf{Q}}

\newcommand{\ysf}{\mathsf{y}}

\newcommand{\Xsf}{\mathsf{X}}
\newcommand{\Fsf}{\mathsf{F}}
\newcommand{\Gsf}{\mathsf{G}}

\newcommand{\Usf}{\mathsf{U}}
\newcommand{\Vsf}{\mathsf{V}}

\newcommand{\usf}{\mathsf{u}}
\newcommand{\vsf}{\mathsf{v}}
\newcommand{\ssf}{\mathsf{s}}
\newcommand{\tsf}{\mathsf{t}}
\newcommand{\wsf}{\mathsf{w}}
\newcommand{\Rbb}{\mathbb{R}}
\newcommand{\Ebb}{\mathbb{E}}

\newcommand{\Scal}{\mathcal{S}}
\newcommand{\Ncal}{\mathcal{N}}
\newcommand{\Gcal}{\mathcal{G}}

\newcommand{\Range}[1]{\mbox{\rm Range}(#1)}
\newcommand{\Trace}[1]{\mbox{\rm Tr}(#1)}

\graphicspath{{../../../code/}}

\newcommand{\blue}[1]{\textcolor{black}{#1}}

\begin{document}
	
\title{Structured random sketching for PDE inverse problems}

\author{Ke Chen
  \and
  Qin Li
  \and
  Kit Newton
  \and
  Stephen J. Wright
}

% \author{Ke Chen\footnote{Affiliation, USA. (\email{kechen@math.utexas.edu}). Work of the this author was partially supported by  NSF Award 1740707.}
%   \and
%   Qin Li\footnote{Mathematics Department, University of Wisconsin-Madison, 1210 W. Dayton St., Madison, WI 53706, USA. (\email{qinli@math.wisc.edu}).  Work of this author was partially supported by .....}
%   \and
%   Kit Newton\footnote{Mathematics Department, University of Wisconsin-Madison, 1210 W. Dayton St., Madison, WI 53706, USA. (\email{k.cole.newton@gmail.com}). Work of this author was partially supported by .....}
%   \and
%   Stephen J. Wright\footnote{Computer Sciences Department, University of Wisconsin, 1210 W. Dayton St., Madison, WI 53706, USA. (\email{swright@cs.wisc.edu}). Work of this author was supported by NSF Awards 1628384, 1634597, and 1740707;  Subcontract 8F-30039 from Argonne National Laboratory; and Award N660011824020 from the DARPA Lagrange Program.}
% }

\begin{abstract}
  For an overdetermined system $\mathsf{A}\mathsf{x} \approx \mathsf{b}$ with $\mathsf{A}$
  and $\mathsf{b}$ given, the least-square (LS) formulation $\min_x \,
  \|\mathsf{A}\mathsf{x}-\mathsf{b}\|_2$ is often used to find an acceptable solution
  $\mathsf{x}$. The cost of solving this problem depends on the dimensions
  of $\mathsf{A}$, which are large in many practical instances. This cost
  can be reduced by the use of random sketching, in which we choose a
  matrix $\mathsf{S}$ with many fewer rows than $\mathsf{A}$ and $\mathsf{b}$, and
  solve the sketched LS problem $\min_x \, \|\mathsf{S}(\mathsf{A}
  \mathsf{x}-\mathsf{b})\|_2$ to obtain an approximate solution to the original
  LS problem. Significant theoretical and practical progress has been
  made in the last decade in designing the appropriate structure and
  distribution for the sketching matrix $\mathsf{S}$. When $\mathsf{A}$ and
  $\mathsf{b}$ arise from discretizations of a PDE-based inverse problem,
  tensor structure is often present in $\mathsf{A}$ and $\mathsf{b}$.  For
  reasons of practical efficiency, $\mathsf{S}$ should be designed to have
  a structure consistent with that of $\mathsf{A}$. Can we claim similar
  approximation properties for the solution of the sketched LS problem
  with structured $\mathsf{S}$ as for fully-random $\mathsf{S}$?  We give
  estimates that relate the quality of the solution of the sketched LS
  problem to the size of the structured sketching matrices, for two
  different structures. Our results are among the first known for
  random sketching matrices whose structure is suitable for use in PDE
  inverse problems.
\end{abstract}

\maketitle

% \renewcommand{\thefootnote}{\fnsymbol{footnote}}
% \footnotetext[1]{Version of \today.}
% \footnotetext[2]{Affiliation, USA. (\email{kechen@math.utexas.edu}). Work of the this author was patially supported by  NSF Award 1740707.}
% \footnotetext[3]{Mathematics Department, University of Wisconsin-Madison, 1210 W. Dayton St., Madison, WI 53706, USA. (\email{qinli@math.wisc.edu}).  Work of this author was partially supported by .....}
% \footnotetext[4]{Mathematics Department, University of Wisconsin-Madison, 1210 W. Dayton St., Madison, WI 53706, USA. (\email{k.cole.newton@gmail.com}). Work of this author was partially supported by .....}
% \footnotetext[5]{Computer Sciences Department, University of Wisconsin, 1210 W. Dayton St., Madison, WI 53706, USA. (\email{swright@cs.wisc.edu}). Work of this author was supported by NSF Awards 1628384, 1634597, and 1740707;  Subcontract 8F-30039 from Argonne National Laboratory; and Award N660011824020 from the DARPA Lagrange Program.}
% \renewcommand{\thefootnote}{\arabic{footnote}}

\section{Introduction} \label{sec:intro}
In overdetermined linear systems (in which the number of linear
conditions exceeds the number of unknowns), the least-squares (LS)
solution is often used as an approximation to the true solution when
the data contains noise. Given the system $\Amat\xmat = \bmat$ where
$\Amat\in\Rbb^{n\times p}$ with $n\gg p$, the least-squares solution
$\xmat^\ast$ is obtained by minimizing the $l^2$-norm discrepancy
between the $\Amat \xmat$ and $\bmat$, that is,
\begin{equation}\label{eqn:LS}
  \min_\xmat \, \|\Amat\xmat-\bmat\|_2\,,\implies \xmat^\ast =  \Amat^\dagger\bmat,\;\; \mbox{where $\Amat^\dagger \defeq (\Amat^\top\Amat)^{-1}\Amat^\top$.}
\end{equation}
The matrix $\Amat^\dagger$ is often called the {\em pseudoinverse}
(more specifically the {\em Moore-Penrose pseudoinverse}) of $\Amat$.
	
The LS method is ubiquitous in statistics and engineering, but large
problems can be expensive to solve.  Aside from the cost of preparing
$\Amat$, the cost of solving for $\xmat^\ast$ is $\mathcal{O}(np^2)$
flops for general (dense) $\Amat$ is prohibitive in large dimensions.
	
We can replace the LS problem with a smaller approximate LS problem by
using {\em sketching}. Each row of the sketched system is a linear
combination of the rows of $\Amat$, together with the same linear
combination of the elements of $\bmat$. This scheme amounts to
defining a sketching matrix $\Smat\in\Rbb^{r\times n}$ with $r\ll n$,
and replacing the original LS problem by
\begin{equation}\label{eqn:LS_ast}
  \min_\xmat \, \|\Smat\Amat\xmat-\Smat\bmat\|_2\,,\implies \xmat_s^\ast =(\Smat\Amat)^\dagger \Smat\bmat\,.
\end{equation}
For appropriate choices of $\Smat$, the solutions of \eqref{eqn:LS}
and \eqref{eqn:LS_ast} are related in the sense that
\begin{equation}\label{eqn:error}
\mbox{ $\| \bmat-\Amat\xmat^\ast\|$ is not too much \blue{smaller} than $\|\bmat-\Amat\xmat^\ast_s\|$.}
\end{equation}
Usually one does not design $\Smat$ directly, but rather draws its entries from
a certain distribution. In such a setup, we can ask
whether~\eqref{eqn:error} holds with high probability. 
	
%Many theoretical results and numerical studies of randomized
%sketching have been presented during the past
%decade~\cite{woodruff2014sketching,clarkson2017low,rokhlin2008fast,meng2013low,nelson2013osnap,sarlos2006improved},
%much of it linked to the Johnson-Lindenstrauss lemma
%\cite{JohL84}. Technically, to a large extent, there are two
%approaches leading to similar results: one approach starts with the
%least square problem directly where the authors proposed two
%conditions the random matrix needs to satisfy for an accurate
%solution with high confidence, and then justified that certain
%choices of random matrices indeed satisfy these two conditions. See
%the original paper in~\cite{Drineas2011}, and a summary
%in~\cite{Mahoney_review}. The other approach more focuses on the
%structure of the spaces spanned by $\Asf$, in which the authors
%argued the space could be approximated by a finite number of vectors
%(so-called $\gamma$-net), which further could be ``embedded" using
%random matrices with high accuracy. See
%review~\cite{woodruff2014sketching} and references therein. This is
%the approach we will be using in this paper. There are many
%variations of the original sketching problem: with some statistical
%assumptions on the perturbation in the right hand side, results could
%be further enhanced~\cite{Raskutti_Mahoney}, and the sketching
%problem is also investigated when other constraints present (such as
%$l_1$ constraints)~\cite{wainwright}.

\blue{ The literature on random sketching is rich. During the past
  decade, many theoretical and numerical studies have
  appeared~\cite{sarlos2006improved,Drineas2011,rokhlin2008fast,woodruff2014sketching,clarkson2017low,meng2013low,nelson2013osnap,diao18a,sun2018tensor,avron2014subspace,Pagh2013,jin2019faster,chi2018randomized,ma2015statistical,liu2018simultaneous},
  with applications in such subjects as stochastic optimization
  \cite{liu2018simultaneous}, $l^p$ regression
  \cite{woodruff2013subspace,Woodruff16,clarkson2017low,meng2013low,Raskutti_Mahoney,rokhlin2008fast,SohlerWoodruff11},
  and tensor
  decomposition~\cite{BIAGIONI2015116,ChengPengLiu16,Battaglino18,Reynolds16,malik2018low}. The
  technical support for these results comes mostly from the Johnson-Lindenstrauss
  lemma~\cite{JohL84}, random matrix
  theory~\cite{vershynin2018high,vershynin2010introduction}, and
  compressed sensing \cite{eldar2012}.} Two important
perspectives have been utilized. One approach starts with the least
squares problem and proposes two conditions for the
random matrix such that an accurate solution can be attained with high
confidence. It is then shown that  certain choices of random matrices
indeed satisfy these two conditions. Instances of this approach can be found in~\cite{sarlos2006improved,Drineas2011,rokhlin2008fast} and
\blue{the reviews~\cite{Mahoney_review,martinsson2020}.} The second
perspective focuses on the structure of the space spanned by
$\Amat$.  It is argued that this space can be approximated by a
finite number of vectors (the so-called $\gamma$-net), which can further
be ``embedded" using random matrices, with high
accuracy; see~\cite{SohlerWoodruff11,Woodruff16,woodruff2013subspace}
and a review \cite{woodruff2014sketching}. We use this second
perspective in this paper.

% To a large extent, these results adopted two perspectives of solving
%the linear regression problems. From a linear algebra angle, one
%proposes conditions that a random matrix needs to satisfy for an
%accurate regression solution and then constructs those random
%matrices as desired. See the original and extension paper
%\cite{sarlos2006improved,Drineas2011,rokhlin2008fast} and summary in
%\cite{Mahoney_review}. From a subspace embedding view, one focuses on
%the structure of the space spanned by $\Amat$ and aims to build an
%embedding for a collection of vectors (so-called $\gamma$-net), which
%further leads to a subspace embedding of the whole space. In this
%case, one considers $l^p$ regression problems and studies the
%property of random linear mappings between $l^p$ spaces. See
%\cite{SohlerWoodruff11,Woodruff16,woodruff2013subspace} and review
%\cite{woodruff2014sketching}.~\kc{added a few woodruff's paper on
%$l^p$ regression or with different construction of random matrices.}
%In this paper we follow the second perspective.

There are many variations of the original sketching problem. With some
statistical assumptions on the perturbation in the right hand side,
results could be further enhanced~\cite{Raskutti_Mahoney}, and the
sketching problem is also investigated when other constraints (such as
$l_1$ constraints) are present; see for example
\cite{wainwright}. \blue{In~\cite{chi2018randomized,Pilanci16,Drineas2011}
  the authors also directly quantify $\|\xsf_s^\ast - \xsf^\ast\|$
  instead of the residual, as in~\eqref{eqn:error}.}
%~\ql{I would like to be very careful in this
%  paragraph -- of course someone is going to get mad here. Maybe
%  Rokhlin? But I honestly think the ones cited here are what we find
%  useful. There are a lot of conferences papers where small deviations
%  are added here and there but I do not think they are relevant to
%  this paper... should we be generous or should we leave it like
%  this?} ~\kc{I don't see why...his works mainly used random sampling
%  to form matrix factorization and not related to sketching and linear
%  regression. I didn't find related sketching papers except those
%  $L^p$ regression ones.}\ql{up to you guys. I am happy with
%  either.}\ql{Okay, we need to add Ward's paper at least...}

In most previous studies, the design of $\Smat$ varies according to
the priorities of the application. For good accuracy with small $r$,
random projections with sub-Gaussian variables are typically
used. When the priority is to reduce the cost of computing the product
$\Smat\Amat$, either sparse or Hadamard type matrices have been
proposed, leading to ``random-sampling'' or FFT-type reduction in cost
of the matrix-matrix multiplication. To cure ``bias'' in the selection
process, leverage scores have been introduced; these trace their
origin back to classical methods in experimental design.
	
In this paper, with practical inverse problems in mind, we consider
the case in which $\Amat$ and $\bsf$ have certain tensor-type
structures. For the sketched system to be formed and solved
efficiently, the random sketching matrix $\Smat$ must have a
corresponding tensor structure. For these tensor-structured sketching
matrices $\Smat$, we ask: What are the requirements on $r$ to achieve
a certain accuracy in the solution $\xmat^\ast_s$ of the sketched
system?

We consider $\Amat$ with the following structure:
\blue{
\begin{equation}\label{eqn:Astructure}
\Amat=\Fmat \ast \Gmat \,,
\end{equation}
where $\ast$ denotes the (column-wise)} \textit{Khatri-Rao product} of the matrices $\Fsf$ and $\Gsf$.
Assuming
$i_1\in\mathcal{I}_1$ and $i_2\in\mathcal{I}_2$, with cardinalities
$n_1=|\mathcal{I}_1|$ and $n_2=|\mathcal{I}_2|$, respectively, the
dimensions of these matrices are
\begin{equation} \label{eqn:FG}
  \Fmat \in\Rbb^{n_1\times p}\,,\quad\Gmat \in\Rbb^{n_2\times
    p}\,,\quad\Amat\in\Rbb^{n\times p}\,,
\end{equation}
where $n=|\mathcal{I}_1\otimes\mathcal{I}_2|= n_1n_2$.
% (Here $|\cdot|$ denotes the cardinality of a set.)
	
By defining $\fmat_j=\Fmat_{:,j} \in \mathbb R^{n_1}$ and $\gmat_j=\Gmat_{:,j}\in \mathbb R^{n_2}$, we can define $\Amat$
alternatively as
\begin{equation} \label{eq:Afg}
  \amat_j \defeq \Amat_{:,j}=\fmat_j\otimes\gmat_j\,,
\end{equation}
where $\amat_j \in \Rbb^n$ denotes the $j$th column of $\Amat$, for
$j=1,2,\dotsc,p$. For vector $\bsf$, we assume that it admits the same tensor structure, that is,
\begin{equation}\label{eqn:bstructure}
	\bsf = \fsf_\bsf \otimes \gsf_\bsf \,, \quad \mbox{for some fixed $\fsf_\bsf \in \Rbb^{n_1}$ and $\gsf_\bsf \in \Rbb^{n_2}$\,.}
\end{equation}

This type of structure comes from the fact that to formulate inverse
problems, one typically needs to prepare both the \emph{forward} and
\emph{adjoint} solutions. Denoting by $\sigma(x)$  the unknown function
to be reconstructed in the inverse PDE problem,  a very typical
formulation is written as a Fredholm integral of the first type:
\begin{equation} \label{eq:fred}
  \int f_{i_1}(x)g_{i_2}(x)\sigma(x) d{x} = \text{data}_{i_1,i_2}\,,
\end{equation}
where $f_{i_1}$ and $g_{i_2}$ solve the forward and adjoint equations
respectively, equipped with boundary/initial conditions indexed by
$i_1$ and $i_2$. Each term on the right-hand side of \eqref{eq:fred} is typically data measured at $i_2$ with input source index $i_1$. To reconstruct $\sigma$, one loops over the entire
list of conditions for $f_{i_1}$ ($i_1 \in \mathcal{I}_1$) and
$g_{i_2}$ ($i_2 \in \mathcal{I}_2$).  The LS formulation
$\min\|\Amat\xmat-\bmat\|_2$ is the discrete version of the Fredholm
integral \eqref{eq:fred}.
	
%especially inverse problems originated in PDEs, usually one needs to
	%compute both the forward equation and the adjoint equation
	%with various sets of parameters and assemble them to
	%formulate a Fredholm integral of the first kind, in which the
	%product of the forward and adjoint equation's solutions plays
	%the role of the Fredholm kernel. The lease square
	%problem~\eqref{eqn:LS} is simply the discrete version of the
	%Fredholm integral, where $p$ stands for the number of
	%parameters to be reconstructed, and $\mathcal{I}_1$ and
	%$\mathcal{I}_2$ collects the possible parameters in the
	%forward and adjoint equations respectively. $\Fmat$ and
	%$\Gmat$ then corrects the information of all possible forward
	%and adjoint solutions (and thus could have infinite many
	%rows).

\blue{This structure imposes requirements on the sketching matrix
$\Smat$. Since $\mathcal{I}_1$ and $\mathcal{I}_2$ contain conditions
for different sets of equations, sketching needs to be performed
within $\mathcal{I}_1$ and $\mathcal{I}_2$ separately. This condition
is reflected by choosing the sketching matrix $\Smat$ to be the \textit{row-wise Khatri-Rao product of $\Pmat$ and $\Qmat$}, that is,}
\[
\Smat_{i,:} = \pmat_i^\top\otimes\qmat_i^\top\,,
\]
where $\pmat_i\in\Rbb^{n_1}$ and $\qmat_i\in\Rbb^{n_2}\,, i = 1,\ldots,p$. The
product $\Smat \Amat$ then has the special form:
\begin{equation} \label{eq:SA}
  (\Smat\Amat)_{i,:} = (\pmat^\top_{i}\Fmat )\circ
  (\qmat^\top_{i}\Gmat)\,,\quad\text{or
    equivalently}\quad(\Smat\Amat)_{i,j} =
  (\pmat_i^\top\fmat_j)(\qmat_i^\top\gmat_j).
\end{equation}
Thus, to formulate the $i$ row in the reduced (sketched) system, we
perform a linear combination of parameters in $\mathcal{I}_1$
according to $\pmat_{i}$ to feed in the forward solver, and a linear
combination of parameters in $\mathcal{I}_2$ according to $\qmat_{i}$
to feed in the adjoint solver, then assemble the results in the
Fredholm integral~\eqref{eq:fred}.
	
With the structural requirements for $\Smat$ in mind, we consider the
following two approaches for choosing $\Smat$.
\begin{itemize}
\item[Case 1:] Generate two random matrices $\Pmat$ and $\Qmat$, of
  size $r_1\times n_1$ and $r_2\times n_2$, respectively, and define
  $\Smat$ to be their tensor product:
  \begin{equation}\label{eqn:case1}
    \Smat = \Pmat\otimes\Qmat\in\Rbb^{r_1 r_2\times n_1 n_2}\,.
  \end{equation}
\item[Case 2:] Generate two sets of $r$ random vectors
  $\{\pmat_i\,,i=1,2,\dotsc,,r\}$ and $\{\qmat_i\,,i=1,2,\dotsc,r\}$, with
  $\pmat_i \in \Rbb^{n_1}$ and $\qmat_i \in \Rbb^{n_2}$ for each $i$,
  and define row $i$ of $\Smat$ to be the tensor product of the
  vectors $\pmat_i$ and $\qmat_i$:
  \begin{equation}\label{eqn:case2}
    \Smat = 
    \frac{1}{\sqrt{r}}
    \begin{bmatrix}
      \pmat^\top_1 \otimes \qmat_1^\top \\
      \vdots \\
      \pmat^\top_r \otimes \qmat_r^\top \\
    \end{bmatrix} \in \Rbb^{r \times n_1 n_2 }
    \,.
  \end{equation}
\end{itemize}
Case 2 gives greater randomness, in a sense, because the rows of
$\Pmat$ and $\Qmat$ are not ``re-used'' as in the first option.

We are not interested in designing sketching matrices of Hadamard
type. In practice, $\Amat$ is often semi-infinite: $\Fmat$ and $\Gmat$
contain all possible forward and adjoint solutions, a set of infinite
cardinality that cannot be prepared in advance. In practice, one can
only obtain the ``realizations'' $\pmat^\top\Fmat$ or
$\qmat^\top\Gmat$ obtained by solving the forward and adjoint
equations with the parameters contained in $\pmat$ and
$\qmat$. Because we use this technique to find $\Smat\Amat$, rather
than computing the matrix-matrix product explicitly,
% (as is more conventional),
there is no advantage to defining $\Smat$ in terms of
Hadamard type random matrices.

\blue{There have been discussions in the sketching literature on
  problems that share our setups, including sketching of matrices
  $\Asf$ with Khatri-Rao product structure. The paper
  \cite{BIAGIONI2015116} presents a tensor interpolative decomposition
  problem which discusses Khatri-Rao product form, but there is not a
  focus on sketching. The paper \cite{sun2018tensor} proposes a
  so-called tensor random projection (TRP), similar to our Case 2
  presented below. However, they mainly obtain sketching of one
  arbitrarily given vector in the space, while we need to sketch the
  entire space. Directly employing their argument in our setting would
  lead to $r=\mathcal{O}(p^8/\varepsilon^2)$, whereas our argument
  suggests that having $r=\mathcal{O}(p^6/\varepsilon)$ is
  sufficient. This point will be discussed further in
  Theorem~\ref{thm:maincase2}.}

\blue{In~\cite{malik2019,jin2019faster} the authors considered the
  fast Johnson-Lindenstrauss Transform (JLT) random matrices and
  showed that the Kronecker product of fast JLT is also a JLT. This
  structure allows embedding of an arbitrarily given vector. For
  embedding vectors that have tensor structure,
  \cite{diao18a,NIPS2019Woodruff} developed \textit{TensorSketch} or
  \textit{CountSketch}, and discussed the efficiency of these
  algorithms in terms of the number of nonzero entries in $\Asf$. All
  these results are highly related to ours, but they all have
  dependences on the ambient space dimension $n$, making them poorly
  suited to our setting, where we consider the possibility of
  $n\to\infty$.}

The rest of the paper is organized as follows. In
Section~\ref{sec:inverse}, we give two examples from PDE-based inverse
problem that give rise to a linear system with tensor structure.
Section~\ref{sec:mainResult} presents classical results on sketching
for general linear regression, and states our main results on
sketching of inverse problem associated with a tensor
structure. Sections~\ref{sec:case1} and \ref{sec:case2} study the two
different sketching strategies outline above. Computational testing
described in Section~\ref{sec:numerical} validates our results.

%% To our knowledge, our results are the first concerning structured
%% sketching for overdetermined problems with tensor-structured
%% coefficient operators.

% \ql{need further check...}
%
% but there was no restriction on the sketching matrix $\Ssf$. In~\cite{jin2019faster}, the authors investigated a fast Johnson-Lindenstrauss Transform for Kronecker matrix product where the sketching matrix if of a specific Fourier transform type. Enforcing the sketching matrix to have a tensor structure is one of the unique features of the current paper.
	
% \paragraph{Notation.}
We denote the range space (column space) of a matrix $\Xmat$ by
$\Range{\Xmat}$.

\section{Overdetermined systems with tensor structure arising from PDE inverse problems}
\label{sec:inverse}

\blue{ Most PDE-based inverse problems, upon linearization, reduce to
  a tensor structured Fredholm integral~\eqref{eq:fred}, which can be
  discretized to formulate a sketching problem.}

\blue{One particularly famous example is Electrical Impedance
  Tomography (EIT), in which we apply voltage strength and measure
  current density at the boundary of some bio-tissues to infer for
  conductivity inside the body. The underlying PDE is a standard
  second order elliptic equation
\begin{equation}\label{eqn:nonlinearEIT}
\begin{aligned}
\nabla_x \cdot\left( \overline{\sigma}(x) \nabla_x \overline{\rho}(x) \right) &= 0\,, \quad &x\in \Omega\,, \\
\rho(x) &= \phi(x)\,, \quad &x\in \partial \Omega\,,
\end{aligned}
\end{equation}
where $\phi(x)$ is the voltage strength applied on the surface of some
bio-tissue, while $\overline{\rho}(x)$, the solution to the PDE, is
the voltage generated throughout the body. The unknown conductivity
$\overline{\sigma}(x)$ will be inferred. The measurements are taken on
the boundary too. In particular, one measures the current density on
the surface of $\Omega$ tested on a testing function $\psi$, as follows:
\begin{equation}\label{eqn:data_phipsi}
	\overline{\text{data}}_{\phi,\psi} = \int_{\partial \Omega } \overline{\sigma}(x) \frac{\partial \overline{\rho}(x)}{\partial n} \psi(x) \rd x\,.
\end{equation}
Here, $\frac{\partial }{\partial n}$ is the normal derivative, with
$n$ being the normal direction pointing out of domain $\Omega$. The
data has two subscripts: $\phi(x)$ is the voltage applied to the
surface and $\psi(x)$ is a testing function that encodes the way
measurements are taken. When the detector is extremely precise, one
can set $\psi(x) = \delta(x-x_0)$ for some $x_0\in \partial \Omega$,
making $\text{data}_{\phi,\psi}$ the current at point $x_0$ when
voltage $\phi$ is applied. With infinite pairs of $\phi$ and $\psi$ in
the experimental setup, EIT seeks to reconstruct
$\overline{\sigma}(x)$. EIT further reduces to the famous Calder\'on
problem when the span of $\phi$ and $\psi$ covers the entire
$H^{1/2}$.  In practice, however, one typically has a rough estimate
of the media $\overline{\sigma}(x)$, termed the background media
$\sigma^\ast(x)$. (For example, most human lungs have the same
structure.) In such situations, one can linearize and reconstruct the
perturbation $\sigma(x)\defeq \overline{\sigma}(x) - \sigma^\ast(x)
\ll 1$. Specifically, suppose that $\rho_1$ solves the following
background forward equation:
\begin{equation}\label{eqn:background}
\begin{aligned}
\nabla_x \cdot\left( \sigma^\ast(x) \nabla_x \rho_1(x) \right) &= 0\,, \quad &x\in \Omega \\
\rho_1(x) &= \phi(x)\,, \quad &x\in \partial \Omega
\end{aligned}
\end{equation}
with the same boundary condition $\phi$ and the given known background
media $\sigma^\ast$. Since both these quantities are known,
$\rho_1(x)$ can be solved ahead of time for any $\phi$. We can also
define the adjoint equation:
\begin{equation}\label{eqn:adjoint}
\begin{cases}
\nabla_x\cdot(\sigma^\ast(x)\nabla_x \rho_2(x)) = 0\,,\;\; & x\in\Omega\\
\rho_2(x)= \psi(x)\,, \;\; & x\in \partial \Omega\,.
\end{cases}
\end{equation}
} \blue{To obtain the Fredholm integral, we take the difference of
  \eqref{eqn:background} and \eqref{eqn:nonlinearEIT} and drop higher
  order terms in $\sigma(x)$ to obtain
\begin{equation}\label{eqn:difference}
\begin{aligned}
\nabla_x\cdot(\sigma^\ast(x)\nabla_x \rho(x)) &= -\nabla_x \cdot(\sigma(x) \nabla_x \rho_1(x))\,,\;\; & x\in\Omega\\
\rho(x) &= 0\,,  \;\; & x\in \partial \Omega
\end{aligned}
\end{equation}
where $\rho(x)\defeq \overline{\rho}(x)-\rho_1(x)$.}  \blue{ With this
  equation multiplied with $\rho_2$ and the adjoint
  \eqref{eqn:adjoint} multiplied with $\rho$, we integrate over
  $\Omega$ and integrating by parts.
  % \footnote{SJW: could you check around here? I am not sure this is all totally right. ~\kc{\eqref{eqn:adjoint} should be multiplied with $\rho$ actually. Now fixed.}}
  The left hand sides
  cancel and the right hand side of~\eqref{eqn:difference} will be
  balancing the boundary terms:
\begin{equation}\label{eqn:inverse_eit2}
\begin{aligned}
\int \nabla_x \rho_1(x) \cdot\nabla_x \rho_2(x)\sigma(x)\rd{x} & = \int_{\partial\Omega} \sigma^\ast \frac{\partial \rho}{\partial n} \psi \rd x + \int_{\partial \Omega} \sigma \frac{\partial \rho_1}{\partial n} \psi \rd x \,.
\end{aligned}
\end{equation}
While the left hand side of this equation is Fredholm integral testing
on $\sigma$ (the conductivity to be reconstructed) with test function
$\nabla_x \rho_1(x) \cdot\nabla_x \rho_2(x)$, the right hand side is
the data that we obtain from measurement
$\text{data}_{\phi,\psi}$. Indeed, since $\overline{\rho}=\rho+\rho_1$
and $\overline{\sigma}=\sigma+\sigma^\ast$, with $\rho\ll 1$ and
$\sigma\ll 1$, 
%the right hand side becomes approximately
the right hand side can be approximated by dropping the higher order term $\int_{\partial \Omega }\sigma\frac{\partial \rho}{\partial n} \psi \rd x$, as follows:
%\begin{equation*}
%\int_{\partial \Omega } {\sigma}^\ast(x) \frac{\partial {\rho}}{\partial n} \psi \rd x+\int_{\partial \Omega } \overline{\sigma}(x) \frac{\partial {\rho_1}}{\partial n} \psi \rd x - \int_{\partial \Omega } \sigma^\ast(x) \frac{\partial \rho_1}{\partial n} \psi \rd x  =\int_{\partial \Omega } \overline{\sigma}(x) \frac{\partial \overline{\rho}}{\partial n} \psi \rd x - \int_{\partial \Omega } \sigma^\ast(x) \frac{\partial \rho_1}{\partial n} \psi \rd x,
%\end{equation*}
\begin{align*}
& \int_{\partial\Omega} \sigma^\ast \frac{\partial \rho}{\partial n} \psi \rd x + \int_{\partial \Omega} \sigma \frac{\partial \rho_1}{\partial n} \psi \rd x \\
=&\int_{\partial \Omega } \overline{\sigma}(x) \frac{\partial \overline{\rho}}{\partial n} \psi \rd x - \int_{\partial \Omega } \sigma^\ast(x) \frac{\partial \rho_1}{\partial n} \psi \rd x-\int_{\partial\Omega} \sigma\frac{\partial\rho}{\partial n}\rd x\\
\approx & \int_{\partial \Omega } \overline{\sigma}(x) \frac{\partial \overline{\rho}}{\partial n} \psi \rd x - \int_{\partial \Omega } \sigma^\ast(x) \frac{\partial \rho_1}{\partial n} \psi \rd x\,,
\end{align*}
% \footnote{SJW: Now it seems to me that the left-hand term in this
%   expression is EXACTLY the right-hand side of \eqref{eqn:inverse_eit2},
%   while the left-hand side is an approximation. So: Should the $=$ in
%   this expression be replaced by $\approx$?  Should we rephrase the
%   phrase before this expression? ~\kc{I changed the phrase above and below this equation.}~\ql{i added another line.}} 
% where we dropped the higher order term $\sigma\partial_n\rho$. 
This expression differs from
$\overline{\text{data}}_{\phi,\psi}$ defined
in~\eqref{eqn:data_phipsi} by $\int_{\partial\Omega }\sigma^\ast \frac{\partial \rho_1}{\partial n} \psi \rd x$, a
pre-computed term, and thus the entire term is known. We finally have
\begin{equation}\label{eqn:inverse_eit}
\int \nabla_x \rho_1(x) \cdot\nabla_x \rho_2(x)\sigma(x)\rd{x}  = \text{data}_{\phi,\psi}\,.
\end{equation}
} \blue{ We emphasize that the $\phi$ dependence comes in through
  $\rho_1$ while the $\psi$ dependence comes in through
  $\rho_2$. These functions represent applied voltage source and
  measuring setup, respectively. If one can provide point source and
  point measurement, $\phi$ and $\psi$ can be as sharp as Dirac-delta
  functions.}

\blue{By varying $\phi$ and $\psi$, one finds infinitely many pairs
  $\{\rho_{1}(\cdot;\phi)\,, \rho_2(\cdot;\psi)\}$, each pair
  providing one data point corresponding to one experiment
  setup. These experimental setup altogether give rise to an
  overdetermined Fredholm integral. More details can be found in
  \cite{Borcea02,Cheney99}.}

\blue{A similar problem arises in optical tomography
  \cite{Arridge1999}. Here we inject light into bio-tissue and take
  measurements of light intensity on the surface, to reconstruct the
  optical properties of the bio-tissue. The formulation is 
\begin{equation}\label{eqn:inverse_op}
  \int \rho_1(x,v)\rho_2(x,v)\sigma(x,v)\rd{x}\rd{v} = \text{data}_{\phi,\psi}\,,
\end{equation}
where $(x,v) \in \Omega \otimes \mathbb{S}$ (where $\Omega$ is the
spatial domain and $\mathbb{S}$ is the velocity domain), and $\rho_i$
are solutions to the forward background radiative transfer equation
and the adjoint equation:
\[
\begin{cases}
  v\cdot\nabla_x \rho_1(x,v)= \sigma^\ast(x,v)\mathcal{L}\rho_1(x,v)\,,\;\; & (x,v)\in\Omega\otimes\mathbb{S}\\
	\rho_1(x,v)= 0\,, \;\; & (x,v) \in {\Gamma_-} 
\end{cases},
\]
and 
\[
\begin{cases}
  -v\cdot\nabla_x \rho_2(x,v)= \sigma^\ast(x,v)\mathcal{L}\rho_2(x,v)\,,\;\; & (x,v)\in\Omega\otimes\mathbb{S}\\
  \rho_2(x,v) = \psi(x,v)\,, \;\; & (x,v) \in {\Gamma_+}
\end{cases}.
\]
} In these equations, $\mathcal{L}$ is a known integral linear
operator on $v$, and $\Gamma_-$ and $\Gamma_+$ are the set collecting incoming and
outgoing boundary coordinates, namely $\Gamma_\pm = \{(x,v): x\in\partial\Omega\,, \pm v\cdot n(x)>0\}$ with $n(x)$ being an outer-normal direction at $x\in\partial\Omega$.
% \footnote{SJW: Is ``coordinates'' the right word? Maybe ``sectors''?
% ~\kc{i don't think ``sectors'' is right here because we also have
% $x$ variable. I add ``boundary" to make it more clear.}~\ql{I am
% happy with either. In PDE world, coordinates have been used a lot
% but I understand this paper is facing a different community. I did
% expand the definition of $\Gamma_\pm$ to try to clarify it.}}
By varying the boundary conditions $\phi$ and $\psi$, one can find
infinitely many solution pairs of
$\{\rho_1(\cdot,\phi),\rho_2(\cdot,\psi)\}$, and collect the
corresponding data in \eqref{eqn:inverse_op}. The inverse Fredholm
integral~\eqref{eqn:inverse_op} can then be solved for $\sigma$.  We
refer to \cite{Chen_2018,Arridge1999} for details of the linearization
procedure.

When $\sigma$ is discretized on $p$ grid points, the reconstruction
problem has the semi-infinite form $\Amat\xmat \approx\bmat$, where
$\xmat \in \Rbb^p$ is the discrete version of $\sigma$ and $\Amat$ and
$\bmat$ have infinitely many rows, corresponding to the infinitely
many instances of $\rho_1$ and $\rho_2$. A fully discrete version can
be obtained by considering $n_1$ values of $\rho_1$ and $n_2$ values
of $\rho_2$, and setting $n=n_1 n_2$ to obtain a problem of the form
\eqref{eqn:LS}.  In the remainder of the paper, we study the sketched
form of this system \eqref{eqn:LS_ast}, for various choices of the
sketching matrix $\Smat$.
	
%%  \eqref{eqn:LS}.
%% \[
%% \Amat\xmat \sim\bmat
%% \]
%% where $\Amat$ collects the product of $\rho_i$ and $\bmat$ collects
%% data, while $\xmat$ is the discrete version of $\sigma$. $\Amat$ is
%% semi-infinite since there are infinite many possible conditions to be
%% imposed for $\rho_i$. In practice we call $\Amat\in\Rbb^{n\times
%%   p}$ with $n=n_1n_2\to\infty$. In this article, we study if
%% conditions for $\rho$ are imposed randomly, namely one applies a
%% random matrix $\Smat$ on $\Amat$, how many rows of $\Amat$ one needs
%% to preserve for a good accuracy with high confidence.

\section{Sketching with tensor structures}\label{sec:mainResult}

We preface our results with a definition of $(\vep,\delta)$-$l^2$ embedding.
\begin{definition}[$(\vep,\delta)$-$l^2$ embedding] \label{def:l2emb}
  Given matrix $\bar\Amat$ and $\vep>0$, let $\Smat$ be a random
  matrix drawn from a matrix distribution $(\Omega ,{\mathcal
    {F}},\Pi)$. If with probability at least $1-\delta$, we have
  \begin{equation}\label{eqn:l2emb}
    \left| \|\Smat \ymat\|^2-\| \ymat \|^2 \right| \leq \vep \| \ymat \|^2, \quad \mbox{for all }\ymat
    \in \Range{\bar\Amat}\,,
  \end{equation}
  then we say that $\Smat$ is an $(\vep,\delta)$-$l^2$ embedding of
  $\bar\Amat$.
\end{definition}
	
%Note that $(\vep,\delta)$-$l^2$ embedding is a property of a certain
%distribution, from which $\Smat$ is drawn. However it is sometimes
%easier to state the random matrix $\Smat$ is an
%$(\vep,\delta)$-$l^2$ embedding of a given matrix $\bar\Amat$,
%assuming $\Smat$ is drawn from distribution $\Pi$. For example, it is
%a convention to state that ``a random Gaussian matrix", instead of
%``the distribution of Gaussian random matrix" is an
%$(\vep,\delta)$-$l^2$ embedding. In the context below, we do not
%distinguish the two.
	
Note that \eqref{eqn:l2emb} depends only on the space
$\Range{\bar\Amat}$ rather than the matrix itself, so we sometimes say
instead that the random matrix $\Ssf$ is an $(\vep,\delta)$-$l^2$
embedding of the linear vector space $\Range{\bar\Amat}$. (We use the
two terms interchangeably in discussions below.)

	%% ~\sw{In the definition you define two concepts: $\vep$
	%% 	sketching and oblivious $(\vep,\delta)$-$l^2$
	%% 	embedding. But then you talk about ``$l^2$ embedding
	%% 	and $(\vep,\delta)$-$l^2$ embedding.''  These latter
	%% 	terms are not the same. Are they synonymous with the
	%% 	terms used in the theorem? If so, we should be MUCH
	%% 	clearer.}  \sw{What is ``$l^2$-embedding''?} \sw{Do
	%% 	you mean to say in this sentence that when you talk
	%% 	about a random matrix being an ``$l^2$ embedding''
	%% 	(whatever that is), you actually mean that its
	%% 	distribution satisfies the oblivious
	%% 	$(\vep,\delta)$-$l^2$ embedding property? Needs to be
	%% 	much clearer.}
	
	%% \sw{In some of the later results, you really do talk about
	%% 	``$l^2$ embedding of the subspace $\Range{\Amat}$''
	%% 	rather than ```$l^2$ embedding of the subspace
	%% 	$\Amat$.'' So should the definition actually say ``of
	%% 	$\Amat$ or of $\Range{\Amat}$'' instead of saying ``of
	%% 	$\Amat$''?}
	
	%% \kc{Steve, I am sorry for the confusing definition and
	%% terminologies. I think it would be better to use the term
	%% $(\vep,\delta)$-$l^2$ embedding of matrix/linear vector
	%% space through the paper. So I delete terms like ``the
	%% $\vep$-sketching", ``oblivious $(\vep,\delta)$-$l^2$
	%% embedding". The paragraph below definition 1 defines in an
	%% informal way the $(\vep,\delta)$-$l^2$ embedding of a
	%% linear vector space. I also edited the rest of paper so we
	%% will always and only see the term $(\vep,\delta)$-$l^2$
	%% embedding.}
	
The $(\vep,\delta)$-$l^2$ embedding property is essentially the only
property needed to bound the error resulting from sketching. It can be
shown that if $\Smat$ is an $(\vep,\delta)$-$l^2$ embedding for the
augmented matrix $\bar\Amat \defeq [\Amat,\bsf]$, then the two
least-squares problems \eqref{eqn:LS} and \eqref{eqn:LS_ast} are
similar in the sense of~\eqref{eqn:error}, as the following result
suggests.

\begin{theorem}\label{thm:obliv_preserv}
  For $\vep,\delta \in (0,1/2)$, suppose that $\Smat$ is an
  $(\vep,\delta)$-$l^2$ embedding of the augmented matrix $\bar \Amat
  \defeq [\Amat,\bmat] \in \Rbb^{n \times (p+1)}$.  Then with
  probability at least $1-\delta$, we have
  \begin{equation*}
    \|\Amat\xmat^\ast_s-\bmat\|^2\leq (1+4\vep)\|\Amat\xmat^\ast-\bmat\|^2\,,
  \end{equation*}
  where $\xmat^\ast$ and $\xmat^\ast_s$ are defined in~\eqref{eqn:LS}
  and~\eqref{eqn:LS_ast}, respectively.
\end{theorem}

The proof of the theorem is rather standard. We simply use the definition of the $(\vep,\delta)$-$l^2$ embedding and the fact that:
%  For any $\xmat \in \Rbb^p$, vector $\ymat\defeq \Amat \xmat - \bmat$
%  is in $ \Range{\bar \Amat}$, therefore with probability greater
%  than $1-\delta$, we have
%  \begin{equation} \label{eq:el2}
%  1-\vep \leq \frac{\| \Smat \ymat \|^2}{\| \ymat\|^2} =
%  \frac{\| \Smat(\Amat \xmat - \bmat)\|^2}{\| \Amat \xmat -\bmat\|^2} \leq 1+\vep\, \quad \mbox{for all $\xmat\in\Rbb^p$}.
%  \end{equation}
%  For $\xmat=\xmat^\ast$, we therefore have
%  \begin{equation}\label{eqn:thm1}
%    (1-\vep) \|\Amat \xmat^\ast -\bmat \|^2 \leq \| \Smat (\Amat \xmat^\ast -
%    \bmat) \|^2 \leq (1+\vep) \|\Amat \xmat^\ast -\bmat \|^2\,.
%  \end{equation}
%  By the definition of $\xmat_s^\ast$ in \eqref{eqn:LS_ast}, and using \eqref{eq:el2} with $\xmat=\xmat_s^\ast$, we have
  \[
  (1-\vep) \| \Amat \xmat_s^\ast - \bmat \|^2 \le \| \Smat (\Amat
  \xmat_s^\ast - \bmat ) \|^2 \leq\| \Smat (\Amat \xmat^\ast -
   \bmat) \|^2 \leq (1+\vep)  \|\Amat \xmat^\ast -\bmat \|^2\,.
  \] 
For $0\leq \vep \leq 1/2$, this leads to
  \[
  \|\Amat \xmat_s^\ast -\bmat \|^2 \leq \frac{1+\vep}{1-\vep} \|\Amat
  \xmat^\ast -\bmat \|^2 \leq (1+4\vep) \|\Amat \xmat^\ast -\bmat
  \|^2\,.
  \]

Given this result, we focus henceforth on whether the various sampling
strategies form an $(\vep,\delta)$-$l^2$ embedding of the augmented
matrix $\bar \Amat =[\Amat,\bmat]$.

%% ~\kc{our main theorems deal with $\Amat$ instead of $\bar\Amat$
%%   ...} \sw{So what should we say in this sentence?} ~\kc{perhaps we
%%   should assume $b$ is close to $Ax$ and use this in Theorem
%%   \ref{thm:obliv_preserv}? I need to think more on this before I
%%   can find an accurate statement.}\sw{This seems like a critical
%%   point for the correctness of the paper --- is it resolved? I'll
%%   read on... Doesn't look like it's resolved. We can't combine
%%   Theorems 3 and 4 with Theorem 1. That's kind of bad isn't it? I
%%   wonder how to fix....}
	
Another theorem that is crucial to our analysis, proved
in~\cite{woodruff2014sketching}, states that Gaussian matrices are
$(\vep,\delta)$-$l^2$ embeddings if the number of rows is sufficiently
large. This result does not consider tensor structure of $\Amat$.
\begin{theorem}[Theorem 2.3 from~\cite{woodruff2014sketching}]\label{thm:sketching}
  Let $\Rmat \in \Rbb^{r\times n}$ be a Gaussian matrix, meaning that
  each entry $\Rmat_{ij}$ is drawn i.i.d. from a normal distribution
  $\mathcal N(0,1)$, and define $\Smat\in \Rbb^{r\times n}$ to be the
  scaled Gaussian matrix defined by
  \[
  \Smat = \frac{1}{\sqrt{r}} \Rmat\,.
  \]
  For any fixed matrix $\Amat \in \Rbb^{n\times p}$ and $\vep,\delta \in
  (0,1/2)$, this choice of $\Smat$ is an
  $(\vep,\delta)$-$l^2$ embedding of $\Amat$ provided that 
  \[
  r \geq \frac{C}{\vep^2}( |\log \delta| + p) \,,
  \]
  where $C>0$ is a constant independent of $\vep$,
  $\delta$, $n$, and $p$.
\end{theorem}
%% ~\kc{I find the $\mathcal{O}$ notation a bit confusing since we
%%   only has the lower bound and it also makes some arguments in the
%%   proof tedious to explain, so I changed above formula in the
%%   normal way.}
The lower bound of $r$ is almost optimal \blue{for the sketched
  regression problem}: the bound is independent of the number of
equations $n$, and grows only linearly in the number of unknowns
$p$. That is, the numbers of equations and unknowns in the sketched
problem \eqref{eqn:LS_ast} are of the same order. The theorem is
proved by constructing a $\gamma$-net for the unit sphere in
$\Range{\Amat}$ and applying the Johnson-Lindenstrauss lemma.
	
Building on the concept of $(\vep,\delta)$-$l^2$ embedding and the
\blue{relationship} between $(\vep,\delta)$-$l^2$ embedding and
sketching (Theorem~\ref{thm:obliv_preserv}), we will study the lower
bound for $r$ (the number of rows needed in the sketching) when the
tensor structure of Case 1 or Case 2 is imposed. Our basic strategy is
to decompose the tensor structure into smaller components to which
Theorem~\ref{thm:sketching} can be applied.
	
We state the results below and present proofs in
Sections~\ref{sec:case1} and \ref{sec:case2} for the two different
cases.

Recall the notation that we defined in Section~\ref{sec:intro}. The
matrices $\Fmat$, $\Gmat$ are defined in~\eqref{eqn:FG} and $\Amat$ is
defined in~\eqref{eq:Afg}. Both $\Fmat$ and $\Gmat$ are assumed to
have full column rank $p$.
%%  the two full rank matrices that generate the tensor structure are
%% denoted \begin{equation}\label{eqn:FG} \Fmat =
%% [\fmat_1,\ldots,\fmat_{p_1}] \in \Rbb^{n_1 \times p_1}\,, \quad
%% \text{and} \quad \Gmat = [\gmat_1,\ldots,\gmat_{p_2}] \in \Rbb^{n_2
%% \times p_2}\,.  \end{equation} The tensor of the two matrices are
%% denoted $\Fmat \otimes \Gmat$,
We need to design the sketching matrix $\Smat$ to
$(\vep,\delta)$-$l^2$ embed $\Range{\bar\Amat}$, the space spanned by
$\{\fsf_\bsf\otimes \gsf_\bsf\} \cup \{ \amat_j \defeq \fmat_j\otimes
\gmat_j\,, j =1,\ldots,p \}$.  In Theorem \ref{thm:maincase1} and
\ref{thm:maincase2}, we construct the $(\vep,\delta)$-$l^2$ embedding
matrix of the Kronecker product $\Fsf\otimes \Gsf$, which
automatically becomes a $(\vep,\delta)$-$l^2$ embedding of its column
submatrix $\Asf$. Moreover, we show in Corollaries \ref{cor:maincase1}
and \ref{cor:maincase2} that these results can be extended to
construct $(\vep,\delta)$-$l^2$ embeddings of the augmented matrix
$\bar\Asf$ by constructing $(\vep,\delta)$-$l^2$ embeddings of the
Kronecker product of the augmented matrices $\bar\Fsf\otimes \bar
\Gsf$, where
\begin{equation} \label{eq:FGbar}
  \bar\Fsf=[\Fsf,\fsf_\bsf]\,, \quad \bar\Gsf=[\Gsf,\gsf_\bsf]\,.
  \end{equation}
	
For Case 1, we have the following result.
\begin{theorem}\label{thm:maincase1}
  Consider $\Smat = \Psf \otimes \Qsf \in \Rbb^{r_1 r_2 \times n_1 n_2}$
  where $\Psf \in \Rbb^{r_1 \times n_1},\Qsf\in \Rbb^{r_2 \times n_2}$
  are independent scaled Gaussian matrices defined by
  \[
  \Pmat \defeq \frac{1}{\sqrt{r_1}} \Rmat \;\; \text{and } \;\; \Qmat
  \defeq \frac{1}{\sqrt{r_2}} \Rmat' \,,\quad
  \mbox{where $\Rmat_{ij},\Rmat_{ij}'$ are i.i.d. normal for all $i,j$}\,.
  \]
  For any given full rank matrices $\Fsf \in \Rbb^{n_1 \times p}$,
  $\Gsf\in \Rbb^{n_2\times p}$, and $\Asf\in \Rbb^{n \times p}$ as in
  \eqref{eqn:FG} and \eqref{eq:Afg}, and $\vep,\delta \in (0,1/2)$,
  the random matrix $\Ssf$ is an $(\vep,\delta)$-$l^2$ embedding of
  $\Fsf\otimes \Gsf$ and $\Amat$  provided that
  \begin{equation}\label{eqn:r_bound1}
    r_i \geq \frac{C}{\vep^2}(|\log \delta| + p)\,, \quad i = 1,2\,,
  \end{equation}
  where the constant $C>0$ is independent of $\vep$, $\delta$, $n_1$,
  $n_2$, and $p$.
\end{theorem}
%% ~\kc{I added $\Fsf \otimes \Gsf$ in the conclusion of the above
%% theorem \ref{thm:maincase1} and theorem \ref{thm:maincase2}.}

\begin{corollary}\label{cor:maincase1}
  Consider the matrices $\Ssf$, $\Fsf$, $\Gsf$, and $\Asf$ from
  Theorem \ref{thm:maincase1}, and assume that the vector $\bsf$ has
  the form \eqref{eqn:bstructure}. Then for given $\vep,\delta \in
  (0,1/2)$, the random matrix $\Ssf$ is an $(\vep,\delta)$-$l^2$
  embedding of the augmented matrix $\bar \Asf \defeq [\Asf,\bsf]$,
  provided that
  \begin{equation}\label{eqn:r_bound1cor}
    r_i \geq \frac{C}{\vep^2}(|\log \delta| + p + 1)\,, \quad i = 1,2\,,
  \end{equation}
  where the constant $C>0$ is independent of $\vep$, $\delta$, $n_1$,
  $n_2$, and $p$.
\end{corollary}
\begin{proof}
  Define the augmented matrices $\bar\Fsf$ and $\bar\Gsf$ as in
  \eqref{eq:FGbar}.  We have that
  \[
  \Range{\bar\Fsf\otimes \bar\Gsf} = \text{Span}\{\Fsf \otimes \Gsf,
  \fsf_1\otimes \gsf_\bsf,\ldots,\fsf_p\otimes \gsf_\bsf, \fsf_\bsf
  \otimes \gsf_1\,\ldots,\fsf_\bsf \otimes \gsf_p, \bsf \} \,.
  \]
\blue{Supposing that $\bar \Fsf$ and $\bar\Gsf$ have full rank}, the
linear subspace Range$(\bar\Asf)$ is a subspace of
Range$(\bar\Fsf\otimes \bar\Gsf)$. By applying
Theorem~\ref{thm:maincase1} to the augmented matrices $\bar\Fsf$ and
$\bar\Gsf$ and using \eqref{eqn:r_bound1cor}, we have that $\Ssf$ is
an $(\vep,\delta)$-$l^2$ embedding of Range$(\bar\Fsf \otimes
\bar\Gsf)$ as well as its subspace Range$(\bar\Asf)$. \blue{Supposing
  that $\bar \Fsf$ is not of full rank but $\bar \Gsf$ is of full
  rank, the subspace Range$(\bar \Fsf\otimes\bar\Gsf)$ is a subspace
  of Range$(\Fsf\otimes \bar\Gsf)$, so similar results can be obtained by
  applying Theorem~\ref{thm:maincase1} to $\Fsf$ and $\bar\Gsf$. Other
  cases regarding the rank of $\bar \Fsf$ and $\bar\Gsf$ can be dealt
  with in the same way.}
\end{proof}

The result for Case 2 is as follows.
\begin{theorem}\label{thm:maincase2}
  Let $\psf_i\in \Rbb^{n_1}$, $\qsf_i\in\Rbb^{n_2}$, $i=1,2,\ldots,r$
  be independent random Gaussian vectors, and define the sketching
  matrix $\Ssf$ to have the form:
  \begin{equation}\label{eqn:formScase2}
    \Smat = \frac{1}{\sqrt{r}} 
    \begin{bmatrix}
      \pmat^\top_1 \otimes \qmat^\top_1 \\
      \vdots\\
      \pmat^\top_r \otimes \qmat^\top_r \\
    \end{bmatrix}
    \in \Rbb^{r\times n_1n_2}\,.
  \end{equation}
  Suppose that $p \ge 6$, and that $\Fsf\in\Rbb^{n_1\times p}, \Gsf\in
  \Rbb^{n_2\times p}$, and $\Asf\in \Rbb^{n\times p}$ are full-rank
  matrices defined as in \eqref{eqn:FG} and
  \eqref{eq:Afg}.  Let $\vep,\delta \in (0,1/2)$. Then the random matrix
  $\Smat$ is an $(\vep,\delta)$-$l^2$ embedding of $\Fsf\otimes \Gsf$
  and $\Amat$ provided that
  \begin{equation}\label{eqn:r_bound2}
    r \geq C \max \left\{ \frac{1}{ \vep} \left(\left| \log
    \delta \right| + p^2 \right)^3 \,, \frac{1}{\vep^{5/2}}
    \right\}\,,
  \end{equation}
  where $C>0$ is a constant independent of $\vep$, $\delta$, $n_1$,
  $n_2$, and $p$.
\end{theorem}

\begin{corollary}\label{cor:maincase2}
  Consider the same matrices $\Ssf$, $\Fsf$, $\Gsf$, and $\Asf$ as in
  Theorem~\ref{thm:maincase2}, with $p \ge 6$, and assume that vector
  $\bsf$ is of the form \eqref{eqn:bstructure}. Then for given
  $\vep,\delta \in (0,1/2)$, the random matrix $\Ssf$ is an
  $(\vep,\delta)$-$l^2$ embedding of the augmented matrix $\bar \Asf
  \defeq [\Asf,\bsf]$ provided that
  \begin{equation}\label{eqn:r_bound2cor}
    r \geq C \max \left\{ \frac{1}{ \vep} \left(\left|
    \log \delta \right| + (p+1)^2 \right)^3 \,,
    \frac{1}{\vep^{5/2}}\right\}\,,
  \end{equation}
  where the constant $C>0$ is independent of $\vep$, $\delta$, $n_1$,
  $n_2$, and $p$.
\end{corollary}
We omit the proof since it is similar to that of
Corollary~\ref{cor:maincase1}. 

Theorems~\ref{thm:obliv_preserv} and \ref{thm:sketching} yield the
fundamental results that, with high probability, for any fixed
overdetermined linear problem, the sketched problem in which $\Smat$
is a Gaussian matrix can achieve optimal residual up to a small
multiplicative error. 
% \sw{At this point there is a disconnect. We go
%  from talking about general Gaussian $\Smat$ to talking about Case 1
%  and Case 2. It's jarring and confusing. Am I understanding
%  correctly? if my understand is correct, I can try to fix.} 
In particular, as will be clear in the proof later, the Case-1
tensor-structured sketching matrix $\Smat=\Pmat\otimes\Qmat$ not only
$(\vep,\delta)$-$l^2$ embeds $\Amat = \Fsf\otimes \Gsf$, but the
number of rows in $\Pmat$ and $\Qmat$ each depends only linearly on
$p$ (see \eqref{eqn:r_bound1}), so that the number of rows in $\Smat$
scales like $p^2$. If the Case-2 sketching matrix is used, the
dependence of $r$ on $p$ and $\vep$ is more complex. Whether this
bound is greater than or less than the bound for Case 1 depends on the
relative sizes of $\vep^{-1}$ and $p$.

\blue{We stress that both bounds show that the number of rows in $\Smat$ is independent of the dimension $n \defeq
n_1 n_2$ of the ambient space. This allows $n$ to be potentially infinity. We also stress that the dependence on $\epsilon$ and $p$ may not be optimal, and the bound may not be tight. As will be seen in the later sections, we have limited understanding of quartic powers of Gaussian random variables, and this confines us obtaining a tighter bound.}
	
	%
	%For the tensor structured inverse problem,
	%Theorem \ref{thm:maincase1} suggests that a tensor-structured
	%sketching matrix $\Smat = \Pmat \otimes \Qmat$ achieves independence
	%of $n$. In particular, it suffices to become an $(\vep,\delta)$-$l^2$
	%embedding of $\Fsf\otimes \Gsf$ with number of rows at the order of
	%\[
	%\mathcal{O}\left(\left( \frac{1}{\vep^2}\left( |\log \delta | + p\right) \right)^2 \right)\approx \mathcal{O}\left(\frac{1}{\vep^4} \left( |\log\delta|^2 + p^2 \right)\right) \,.
	%\]
	%In comparison, Theorem 4 suggests another tensor-structured sketching
	%matrix that also achieves independence of $n$ but with the number of
	%rows at a different scales
	%\[
	%\mathcal{O}\left(\max\left\{ \frac{1}{\vep}(|\log\delta|^3 + p^6) \,, \frac{p^{3/2}}{\vep^{5/2}} 	\right\}\right) \,.
	%\]
	%Depending on particular choice of $(\vep,\delta)$ and $p$, one
	%sketching strategy may require less number of rows than that of other
	%strategy. In particular, for small $p$ and large $\delta$, Case 2
	%sketching matrix requires $\mathcal{O}(\frac{1}{\vep})$ rows to
	%achieve an $\vep$ accuracy in the residue. This is observed in the
	%numerical test part. 
	%%%%%%%%%%%%%%%%%%%%%%%%%%%%%%%%%%%%%%%%%%%%%%
	
\section{Case 1: Proof of Theorem~\ref{thm:maincase1}}\label{sec:case1}

In this section we present the proof of
Theorem~\ref{thm:maincase1}. We start with technical results.

\begin{lemma}\label{lem:1}
  Consider natural numbers $r_2$, $n_1$, and $n_2$, and assume that a
  random matrix $\Qsf\in \Rbb^{r_2\times n_2}$ is an
  $(\vep,\delta)$-$l^2$ embedding of $\Rbb^{n_2}$, meaning that with
  probability at least $1-\delta$, $\Qsf$ preserves $l^2$ norm with
  $\vep$ accuracy, that is,
  \[
  \left| \| \Qsf \xsf \|^2 - \|\xsf \|^2 \right| \leq \vep \|\xsf\|^2 \,,\quad \mbox{for all $\xsf\in\Rbb^{n_2}$.}
  \]
  Then the Kronecker product $\Idsf_{n_1} \otimes \Qsf$ is an
  $(\vep,\delta)$-$l^2$ embedding of $\Rbb^{n_1 n_2}$. Similarly, if
  $\Qsf\in \Rbb^{r_1\times n_1}$ is an $(\vep,\delta)$-$l^2$ embedding
  of $\Rbb^{n_1}$, then $\Qsf\otimes\Idsf_{n_2}$ is an
  $(\vep,\delta)$-$l^2$ embedding of $\Rbb^{n_1 n_2}$.
\end{lemma}
\begin{proof}
  The proof for the two statements are rather similar, so we prove
  only the first claim.
  
  Any $\xsf \in \Rbb^{n_1n_2}$ can be written in the following form
  \[	\xsf = 	\begin{bmatrix}
    \xsf_1 \\
    \vdots\\
    \xsf_{n_1}
  \end{bmatrix}\,, \quad \mbox{where  $\xsf_i \in \Rbb^{n_2}$, $i=1,2,\dotsc,n_1$.} 
  \]
   Then
  \[
  (\Idsf_{n_1} \otimes \Qsf) \xsf = 
  \begin{bmatrix}
    \Qsf & & \\
    & \ddots & \\
    & & \Qsf
  \end{bmatrix}
  \begin{bmatrix}
    \xsf_1 \\
    \vdots\\
    \xsf_{n_1}
  \end{bmatrix}
  =
  \begin{bmatrix}
    \Qsf\xsf_1 \\
    \vdots\\
    \Qsf\xsf_{n_1}
  \end{bmatrix}\,.
  \]
Thus, we have
  \begin{equation}\label{eqn:1}
    \| (\Idsf_{n_1} \otimes \Qsf) \xsf \|^2 = \sum_{i=1}^n \| \Qsf \xsf_i \|^2\,, \quad \| \xsf\|^2 = \sum_{i=1}^n \| \xsf_i\|^2\,.
  \end{equation}
  Since $\Qsf$ is an $(\vep,\delta)$-$l^2$ embedding of $\Rbb^{n_2}$,
  then with probability at least $1-\delta$, for all
  $\xsf_i\in\Rbb^{n_2}$, we have
  \begin{equation}\label{eqn:2}
    \left| \| \Qsf \xsf_i \|^2 - \|\xsf_i \|^2 \right| \leq \vep \|\xsf_i\|^2 \,, \quad \mbox{for all $i = 1,2,\dotsc,n_1$.}
  \end{equation}
  By using this bound in~\eqref{eqn:1}, with probability at least
  $1-\delta $, we have for all $\xsf \in \Rbb^{n_1 n_2}$ that 
  \[
    \left| \| (\Idsf_{n_1} \otimes \Qsf) \xsf \|^2 - \| \xsf\|^2  \right|  \leq \sum_{i=1}^n \left| \| \Qsf \xsf_i \|^2 - \| \xsf_i\|^2 \right| \leq \vep \sum_{i=1}^n  \| \xsf_i\|^2= \vep \| \xsf \|^2\,,
  \]
  so that $(\Idsf_{n_1} \otimes \Qsf)$ is an $(\vep,\delta)$-$l^2$
  embedding of $\Rbb^{n_1n_2}$, as claimed.
\end{proof}

The following corollary extends the previous result and discusses the
embedding property of $\Psf\otimes\Qsf$.

%The following corollary is useful in showing that the Kronecker
%product $\Psf \otimes \Qsf \in \Rbb^{r_1r_2\times n_1 n_2}$ embeds the full vector space $\Rbb^{n_1 n_2}$, so that $\Psf \otimes \Qsf$ satisfies the assumption in Theorem~\ref{thm:maincase1}, for
%suitable choices of the parameters $\vep$ and $\delta$.

\begin{corollary}\label{cor:1}
  Assume two random matrices $\Psf\in\Rbb^{r_1\times n_1}$ and $\Qsf
  \in \Rbb^{r_2\times n_2}$ are $(\vep,\delta)$-$l^2$ embeddings of
  $\Rbb^{n_1}$ and $\Rbb^{n_2}$, respectively. Then the Kronecker
  product $\Psf \otimes \Qsf \in \Rbb^{r_1r_2\times n_1 n_2}$ is an
  $\left(\vep(2+\vep),2\delta \right)$-$l^2$ embedding of $\Rbb^{n_1
    n_2}$.
\end{corollary}
\begin{proof}
  Noting that (see \eqref{eqn:kron1} in Appendix~\ref{sec:app1}),
  \[
  \Psf\otimes \Qsf = (\Psf \otimes \Idsf_{r_2})(\Idsf_{n_1} \otimes \Qsf)\,,
  \]
we have
  \[
    \| (\Psf\otimes \Qsf )\xsf \|^2  = \| (\Psf \otimes \Idsf_{r_2})(\Idsf_{n_1} \otimes \Qsf) \xsf \|^2 = \|(\Psf\otimes\Idsf_{r_2})\ysf\|^2\,,
  \]
  where $\ysf \defeq (\Idsf_{n_1} \otimes \Qsf) \xsf$.
  
  Denote by $(\Omega_1 ,{\mathcal {F}_1},\Pi_1) $ and $(\Omega_2
  ,{\mathcal {F}}_2,\Pi_2)$ the probability triplets for $\Psf$ and
  $\Qsf$, respectively. Since $\Psf$ is an $(\vep,\delta)$-$l^2$
  embedding of $\Rbb^{n_1}$, we have with probability at least $1-\delta$ in
  $\Pi_1$ that 
  \[
  \left|\|  (\Psf\otimes \Idsf_{r_2}) \ysf \|^2 - \| \ysf \|^2 \right| \leq \vep \| \ysf \|^2\,.
  \]
  Similarly, with probability at least $1-\delta$ for the choice of
  $\Qsf$ in $\Pi_2$, we have
  \[
  \left| \| (\Idsf_{n_1} \otimes \Qsf) \xsf \|^2 - \| \xsf\|^2  \right| \leq  \vep \| \xsf \|^2\,, \quad \mbox{for all $\xsf\in \Rbb^{n_1 n_2}$.}
  \]
  Combining the two inequalities, we have with probability at least
  $1-2\delta$ in the joint probability space of $\Pi_1$ and $\Pi_2$
  that the following is true for all $\xsf\in\Rbb^{n_1n_2}$:
  \begin{align*}
    \left| \| (\Psf\otimes \Qsf )\xsf \|^2  - \|  \xsf \|^2 \right| &\leq \left|  \|  (\Psf\otimes \Idsf_{r_2} )\ysf \|^2 - \|\ysf \|^2 \right| + \left| \| (\Idsf_{n_1} \otimes \Qsf) \xsf \|^2 - \| \xsf\|^2  \right|  \\
    & \leq \vep  \|\ysf \|^2 + \vep \| \xsf \|^2\\
    & = \vep  \| (\Idsf_{n_1} \otimes \Qsf) \xsf \|^2 + \vep \| \xsf \|^2\\
    & \leq \vep(2+\vep ) \| \xsf \|^2 \,.
  \end{align*}
  This concludes the proof.
\end{proof}

Now we are ready to show the proof of Theorem \ref{thm:maincase1},
obtained by applying Theorem~\ref{thm:sketching} to
Corollary~\ref{cor:1}.
\begin{proof}[Proof of Theorem \ref{thm:maincase1}]
  For any vector $\ysf$ in the span of $\Fsf\otimes
  \Gsf$, we can write
  \[
  \ysf =( \Usf_\Fsf \otimes \Usf_\Gsf ) \xsf \,, \quad \text{for some } \xsf \in \Rbb^{p^2}\,,
  \]
  where $\Usf_\Fsf \in \Rbb^{n_1 \times p}$ and $\Usf_\Gsf\in
  \Rbb^{n_2 \times p}$ collect the left singular vectors of matrices
  $\Fsf$ and $\Gsf$, respectively. By applying \eqref{eqn:kron1} from
  Appendix~\ref{sec:app1}, we have
  \[
  ( \Usf_\Fsf \otimes \Usf_\Gsf ) = (\Usf_\Fsf \otimes \Idsf_{n_2})(
  \Idsf_p \otimes \Usf_\Gsf).
  \]
  It is easy to see that the matrix $\Idsf_p \otimes \Usf_\Gsf$ has
  orthonormal columns, so it is an isometry. The matrices $\Usf_\Fsf
  \otimes \Idsf_{n_2}$ and $\Usf_\Fsf \otimes \Usf_\Gsf$ are
  isometries for the same reason. As a consequence, we have $\| \ysf
  \|^2 = \| \xsf \|^2$. From \eqref{eqn:kron1} in
  Appendix~\ref{sec:app1}, we have by defining $\widetilde{\Psf}
  \defeq \Psf \Usf_\Fsf \in \Rbb^{r_1 \times p}$ and $\widetilde{\Qsf}
  \defeq \Qsf \Usf_\Gsf\in \Rbb^{r_2 \times p}$ that
  \begin{equation} \label{eq:xu8}
    \Ssf \ysf = (\Psf \otimes \Qsf )(\Usf_\Fsf \otimes \Usf_\Gsf) \xsf = (\Psf \Usf_\Fsf) \otimes (\Qsf \Usf_\Gsf) \xsf = (\widetilde{\Psf} \otimes \widetilde{\Qsf}) \xsf\,.
  \end{equation}
  Due to the orthogonality of $\Usf_\Fsf$ and $\Usf_\Gsf$, the random
  matrices $\widetilde{\Psf}$ and $\widetilde{\Qsf}$ are also
  independent Gaussian matrices \blue{with i.i.d. entries}. According to
  Theorem~\ref{thm:sketching}, for any pair
  $\tilde{\vep},\tilde{\delta}\in (0,1/2)$, by choosing $r_i$ to
  satisfy
  \begin{equation}\label{eqn:r_i_case1}
    r_i \geq \frac{C}{\tilde{\vep}^2} (|\log\tilde{\delta} | + p)\,, \quad i=1,2\,,
  \end{equation}
  we have that $\widetilde{\Psf}$ and $\widetilde{\Qsf}$ are both
  $(\tilde{\vep},\tilde{\delta})$-$l^2$ embeddings of
  $\Rbb^{p}$. Thus, from Corollary~\ref{cor:1}, the tensor product
  $(\widetilde{\Psf} \otimes \widetilde{\Qsf})$ is an
  $\left(\tilde{\vep}(2+\tilde{\vep}),2\tilde{\delta} \right)$-$l^2$
  embedding of $\Rbb^{p^2}$, meaning that  with probability at least
  $1-2\tilde{\delta}$, we have
  \[
  \left|\|  (\widetilde{\Psf} \otimes \widetilde{\Qsf}) \xsf \|^2 - \| \xsf \|^2 \right| \leq \tilde{\vep} (2+\tilde{\vep}) \| \xsf \|^2\,,\quad \text{for all } \xsf \in \Rbb^{p^2}\,.
  \]
Recalling $\| \xsf \|^2 = \| \ysf \|^2$ and \eqref{eq:xu8}, we have that
%%   \[
%%   \Ssf \ysf = (\widetilde{\Psf} \otimes \widetilde{\Qsf}) \xsf\quad \text{and}\quad \,,
%%   \]
%% we have
  \[
  \left|\|  \Ssf \ysf  \|^2 - \| \ysf\|^2 \right| \leq \tilde{\vep} (2+\tilde\vep) \| \ysf \|^2 \,,\quad \text{for all } \ysf \in \text{Span}\{\Fsf \otimes \Gsf \} \,.
  \]
  By defining $\vep = \tilde{\vep} (2+ \tilde{\vep})$ and $\delta = 2
  \tilde{\delta}$, we have
  \[
  \tilde{\vep} = \frac{\vep}{\sqrt{1+\vep}+1} \,, \quad \text{and } \quad \tilde{\delta} = \frac{\delta }{2}\,.
  \]
  Note that if $\vep$ and $\delta$ are in $(0,1/2)$, then $\tilde\vep$
  and $\tilde\delta$ are also in this interval, so
  \eqref{eqn:r_i_case1} applies. By substituting
  into~\eqref{eqn:r_i_case1} we obtain 
  \[
  r_i \geq \frac{C}{\vep^2} (|\log\delta | + p)\,, \quad i=1,2\,.
  \]
  The constant $C$ here is different from the value
  in~\eqref{eqn:r_i_case1} but can still be chosen independently of
  $\vep$, $\delta$, $n_1$, $n_2$, and $p$. We conclude that
  $\Smat=\Psf\otimes\Qsf$ is an $(\vep,\delta)$-$l^2$ embedding of
  $\Fsf\otimes \Gsf$ and thus also an $(\vep,\delta)$-$l^2$ embedding
  of $\Asf$.
\end{proof}

\section{Case 2: Proof of Theorem~\ref{thm:maincase2}}\label{sec:case2}

In this section we investigate Case-2 sketching matrices, which have
the form~\eqref{eqn:formScase2}.
	
We prove Theorem~\ref{thm:maincase2} in two major steps. First, in
Section~\ref{sec:case2_y}, we investigate the accuracy and probability
of embedding any given vector $\ysf \in\text{Span}\{\Fsf\otimes
\Gsf\}$. Second, in Section~\ref{sec:case2_net}, we extend this study
to deal with the whole space $\text{Span}\{\Fsf\otimes\Gsf\}$. To do
so, we first build a $\gamma$-net over the unit sphere in
$\text{Span}\{\Fsf\otimes \Gsf\}$ so that we can ``approximate'' the
space using a finite set of vectors. By adjusting $\vep$ and $\delta$,
one not only preserves the norm, but also the angles between the
vectors on the net. We then map the net back to the space to show that
$\Smat$ preserves the norm of the vectors in the whole space. This
standard technique is used in \cite{woodruff2014sketching} to prove
their Theorem~\ref{thm:sketching}.

\subsection{Embedding a given vector}\label{sec:case2_y}

We establish the following result, whose proof appears at the end of
the subsection.
\begin{proposition}\label{prop:one_y}
  Given two full rank matrices $\Fsf$ and $\Gsf$ as in \eqref{eqn:FG}
  and $\vep \in (0,1/2)$, let $\Ssf\in\mathbb{R}^{r\times n_1n_2}$
  have the form of~\eqref{eqn:formScase2}, with $\pmat_i$ and
  $\qmat_i$, $i=1,2,\dotsc,r$ being i.i.d. Gaussian vectors. Then for
  any fixed $\ysf \in\text{Span}\{\Fsf\otimes\Gsf\}$, we have that
  \[
  \Pr\left( \left| \| \Ssf \ysf\|^2 - \| \ysf \|^2 \right| > \vep \|
  \ysf \|^2 \right) \leq 5r \exp \left( \frac34 p^{1/2} \right)
  \exp\left( -\frac12 r^{1/3} \vep^{1/3}\right)\,,
  \]
  provided that
  \[
  r \geq 8 \cdot 3^{3/2} \cdot \max\{ \vep^{-5/2}, p^{3/2} \vep^{-1}\}
  \,.
  \]
\end{proposition}

Essentially, this proposition says that $\Smat$ is an $(\vep,5r \exp
\left( (3/4)p^{1/2} \right) \exp\left( -(1/2) r^{1/3}
\vep^{1/3}\right))$-$l^2$ embedding of any fixed
$\ymat\in\text{Span}\{\Fsf\otimes\Gsf\}$. The contribution from the
factor $ \exp\left( -(1/2) r^{1/3} \vep^{1/3}\right)$ is small when
$r$ is large.

We start with several technical lemmas.
Lemma~\ref{lem:SyDistribution} identifies $\left| \| \Ssf \ysf\|^2 -
\| \ysf \|^2 \right|/\|\ysf\|$ with a particular type of random
variable; we discuss the tail bound for this random variable in
Lemma~\ref{lem:zetaSampleTail}. Lemma~\ref{lem:zetabound} contains
some crucial estimates to be used in Lemma~\ref{lem:zetaSampleTail}.

\begin{lemma}\label{lem:SyDistribution}
  Given two full rank matrices $\Fsf$ and $\Gsf$ as in \eqref{eqn:FG},
  consider $\Ssf$ defined as in \eqref{eqn:formScase2}. Then there
  exists a diagonal positive semi-definite matrix $\Sigma$ with
  $\Trace{\Sigma^2} = 1$ so that for any $\ysf \in\text{Span}\{\Fsf\otimes
  \Gsf\}$ with $\| \ysf\|=1$, we have
  \begin{equation*}
    \| \Ssf \ysf\|^2 \dsim \frac{1}{r}\sum_{i=1}^r \zeta_i^2 \,, \quad \text{where } \zeta_i \defeq \xi_i^\top \Sigma \eta_i\,,
  \end{equation*}
  where $\dsim$ denotes equal in distribution and $\xi_i,\eta_i \in
  \Rbb^p$ are independent Gaussian vectors drawn from
  $\mathcal{N}(0,\Idsf_p)$.
\end{lemma}	
\begin{proof}
  From \eqref{eqn:formScase2} we have
  \[
  \Ssf \ysf = \frac{1}{\sqrt{r}} 
  \begin{bmatrix}
    (\psf_1^\top \otimes \qsf_1^\top) \ysf \\
    \vdots\\
    (\psf_r^\top \otimes \qsf_r^\top) \ysf \\
  \end{bmatrix}\implies \| \Ssf\ysf\|^2 = \frac{1}{r} \sum_{i=1}^r \zeta_i^2\,,
  \]
  where $\zeta_i \defeq (\psf_i^\top\otimes \qsf_i^\top )\ysf$. Since
  $\psf_i$ and $\qsf_i$ are independent Gaussian vectors, all random
  variables $\zeta_i$, $i=1,2,\dotsc,r$, are drawn i.i.d. from the
  same distribution.
  
  We consider now the behavior of $\zeta \defeq (\psf^\top\otimes
  \qsf^\top )\ysf $ for Gaussian vectors $\psf$ and $\qsf$. Notice
  that for any $\ysf \in \Rbb^{n_1n_2}\in\text{Span}\{\Fsf\otimes
  \Gsf\}$, there exists $\xsf \in \Rbb^{p^2}$ such that
  \[
  \ysf = (\Usf_\Fsf\otimes \Usf_\Gsf) \xsf\,, \quad \text{with } \|\xsf\| = 1\,,
  \]
  where $\Usf_\Fsf$ and $\Usf_\Gsf$ collect the left singular vectors
  of $\Fsf$ and $\Gsf$, respectively. We thus obtain from
  \eqref{eqn:kron1} that
  \[
  \zeta =(\psf^\top\otimes \qsf^\top )\ysf  =
  (\psf^\top \otimes \qsf^\top) (\Usf_\Fsf\otimes
  \Usf_\Gsf)\xsf = \left((\psf^\top \Usf_\Fsf) \otimes
  (\qsf^\top\Usf_\Gsf) \right) \xsf = 
  (\tilde{\psf}^\top \otimes \tilde{\qsf}^\top) \xsf\,,
  \]
  where $\tilde{\psf} \defeq \Usf_\Fsf^\top \psf \in \Rbb^p$ and
  $\tilde{\qsf} \defeq \Usf_\Gsf^\top \qsf \in \Rbb^{p}$ are
  i.i.d. Gaussian vectors as well. By applying~\eqref{eqn:kron2} and \eqref{eqn:kron3}, we
  obtain
  \begin{equation} \label{eq:si1}
  \zeta = (\tilde{\psf}^\top\otimes \tilde{ \qsf}^\top )\xsf
  =  \tilde{\qsf}^\top \textbf{Mat}(\xsf)
  \tilde{\psf}\,,
  \end{equation}
  where $\textbf{Mat}(\xsf) \in \Rbb^{p \times p}$ is the matricization of $\xsf$, discussed in Appendix A.  By using the
  singular value decomposition $\textbf{Mat}(\xsf) = \Usf \Sigma
  \Vsf^\top$, we obtain
  \[
  \Trace{\Sigma^2} = \| \textbf{Mat}(\xsf) \|^2_F = \| \xsf \|^2 = 1 \,,
  \]
  where $\| \cdot \|_F$ denotes the Frobenius norm of a matrix.
  By substituting into  \eqref{eq:si1}, we obtain
  \[
  \zeta = (\Usf^\top \tilde{\qsf} )^\top \Sigma \Vsf^\top \tilde{\psf}
  = \xi^\top \Sigma \eta \,,
  \]
  where
  \[
  \xi \defeq \Usf^\top \tilde{\qsf} \in \Rbb^{p} \;\; \text{and}
  \;\; \eta \defeq \Vsf^\top \tilde{\psf} \in \Rbb^{p}
  \]
  are again i.i.d. Gaussian vectors in $\Rbb^p$. This completes the
  proof.
\end{proof}

\begin{lemma}\label{lem:zetabound}
For any fixed diagonal semi-positive definite matrix $\Sigma\defeq
\text{diag}\{\sigma_1,\ldots,\sigma_p\}$ such that
$\Trace{\Sigma^2}=1$, define the random variable $\zeta$ to be
$
\zeta \defeq \xi^\top \Sigma  \eta \,,
$
 with $\xi$ and $\eta$ being i.i.d. random Gaussian vectors with $p$
 components. Then $\zeta$ satisfies the following properties:
 \begin{enumerate}
 	\item[1.]  	\begin{equation}\label{eqn:zeta2tail}
 	\Pr\left( |\zeta| >  t \right) \leq 
 	\begin{cases}
 	2 \exp \left( -\frac{(t-\sqrt{p})^2}{4\sqrt{p}} \right) & \;\; \mbox{if $\sqrt{p} \leq t \leq 2\sqrt{p}$} \\
 	2 \exp \left( -\frac{(2t-3\sqrt{p})}{4} \right) & \;\; \mbox{if $t\geq  2\sqrt{p}$}\\
 	\end{cases}
 	\,,
 	\end{equation}
 	\item[2.]  \begin{equation}\label{eqn:Ebbzeta2and4}
 		\Ebb\left[\zeta^2\right] = 1 \quad \text{and} \quad \Ebb\left[\zeta^4 \right] \leq 9 \,,
 	\end{equation}
 	\item[3.]  \begin{equation}\label{eqn:zetavariance}
 	\Ebb\left[ \left(|\zeta|^2 - \Ebb\left[\zeta^2\right] \right)^2 \right] \leq 8\,.
 	\end{equation}
 \end{enumerate}
\end{lemma}

\begin{proof}
  For any $s>0$ and $t\geq 0$, we apply Markov's inequality to derive
  \begin{equation}\label{eqn:zetatail}
    \Pr\left( \zeta > t \right) = \Pr\left( e^{s
      \zeta } > e^{st}\right) \leq e^{-st}
    \Ebb\left[ \exp\left( s\xi^\top \Sigma \eta
      \right) \right] \,.
  \end{equation}
  %%%%%%%%%%%%%%%%%%%%%%%%%%%%%%%%%%%%%%%%
  Noting that $2\xi^\top \Sigma \eta \leq \|\Sigma^{1/2}\xi\|^2 + \|
  \Sigma^{1/2}\eta\|^2$, we use the independence of $\xi$ and $\eta$
  to deduce that
  \begin{equation}\label{eqn:holder}
    \Ebb\left[ \exp\left( s\xi^\top \Sigma \eta \right) \right] \leq
    \Ebb\left[ \exp\left( (s/2)(\|\Sigma^{1/2}\xi\|^2 + \|
      \Sigma^{1/2}\eta\|^2 )\right)\right] = \Ebb\left[ e^{(s/2)\|
        \Sigma^{1/2}\xi\|^2}\right] \Ebb\left[ e^{(s/2)\|\Sigma^{1/2}
        \eta\|^2}\right] \,.
  \end{equation}
  For the first term on the right-hand side of \eqref{eqn:holder}, using
  independence of the $\xi_i$ and the concave Jensen's inequality, we have that 
    \[
      \Ebb\left[ e^{(s/2)\|\Sigma^{1/2}\xi\|^2 }\right] = \Ebb\left[
        \exp\left(\frac{s}{2}\sum_{i=1}^p \sigma_i
        \xi_i^2\right)\right] =\prod_{i=1}^{p} \Ebb\left[
        e^{s\xi_i^2{\sigma_i}/2}\right] \leq \prod_{i=1}^{p}
      \left(\Ebb\left[ e^{s\xi_i^2/2}\right]\right)^{\sigma_i},
      \]
  where we used $0\leq \sigma_i\leq 1$, $i=1,2,\dotsc,r$ to apply the
  concave Jensen's inequality, and $\xi_i \sim
  \mathcal{N}(0,1)$. According to
  Proposition~\ref{prop:gaussian_square} (see
  Appendix~\ref{sec:app2}), $\xi^2_i-1$ is a sub-exponential random
  variable with parameters $(2,4)$. Thus from \eqref{eqn:sub_exp},
  with $\lambda=2$, $b=4$, and $s$ replaced by $s/2$, we have
  \begin{align*}
    \Ebb\left[ e^{(s/2)\|\Sigma^{1/2}\xi\|^2 }\right]
    & \le  \prod_{i=1}^{p} \left(\Ebb_{\xi} \left[ e^{s\xi^2/2}\right]\right)^{\sigma_i}
    = \left(\Ebb_{\xi} \left[ e^{s\xi^2/2}\right] \right)^{\Trace{\Sigma}} \\
    & = \left(e^{s/2}\Ebb_{\xi} \left[ e^{s(\xi^2-1)/2} \right] \right)^{\Trace{\Sigma}} \leq e^{ (s^2 +s)\Trace{\Sigma}/2 }\,,\;\; \mbox{for $0<s<1/2$.}
  \end{align*}
  %% Here $\xi_i$ is the $i$-th component of the vector $\xi$, a
  %% Gaussian random variable. The last inequality comes from the fact
  %% that $\xi^2_i-1$ is a sub-exponential random variable with
  %% parameter $(2,4)$.
  Since, by H\"older's inequality, we have
  \[
  \Trace{\Sigma} = \sum_{i=1}^p \sigma_i \leq \left(\sum_{i=1}^p \sigma_i^2 \right)^{1/2} \sqrt{p} = \sqrt{p}\,,
  \]
  it follows that
  \[
  \Ebb\left[ e^{(s/2)\|\Sigma^{1/2}\xi\|^2 }\right]\leq e^{(s^2+s)\sqrt{p}/2}\,, \;\; \mbox{for $s \in (0,1/2)$.}
  \]
  The same bound holds for second term on the right-hand side of
  \eqref{eqn:holder}. When we substitute these bounds into \eqref{eqn:zetatail} and \eqref{eqn:holder}, we obtain
  \[
  \Pr\left( \zeta> t \right) \leq \exp\left(\sqrt{p}s^2 -(t-\sqrt{p})s  \right)\,.
  \]
  By minimizing the right-hand side over $s \in [0,1/2]$, we obtain
  \[
  \Pr\left( \zeta> t \right) \leq 
  \begin{cases}
    e^{-\frac{(t-\sqrt{p})^2}{4\sqrt{p}}} & \;\; \mbox{if $\sqrt{p} \leq t \leq 2\sqrt{p}$} \\
    e^{-\frac{(2t-3\sqrt{p})}{4}} & \;\; \mbox{if $t\geq  2\sqrt{p}$}
  \end{cases}
  \,.
  \]
  Due to symmetry, we have the same bound for $\Pr\left( \zeta<- t
  \right) $, so \eqref{eqn:zeta2tail} follows.
%%%%%%%%%%%%%%%%%%%%%%%%%%%%%%%%%%%%%%%%%%%%	
%	\blue{Now apply lemma 6.2.2 from \cite{vershynin2018high}, there exists absolute positive constants $c$ and $C$ such that
%	\[
%	\Ebb\left[ \exp(s \xi^\top \Sigma \eta )\right] \leq \exp\left( C s^2 \right)
%	\]	
%	for $ |s| \leq \frac{c}{\sigma_1}$. Plug this into \eqref{eqn:zetatail}, we have
%	\[
%	\Pr(\zeta > t) \leq \exp(Cs^2 - st)\,, \forall |s| \leq \frac{c}{\sigma_1}\,.
%	\]
%	Now minimize the right-hand side in terms of $s$, we have
%	\begin{equation}
%		\Pr(\zeta>t) \leq 
%		\begin{cases}
%		\exp\left(-\frac{t^2}{4C}\right)\,, \quad &\text{if  } 0\leq t \leq \frac{2cC}{\sigma_1}\\
%		\exp\left(- \frac{ct}{\sigma_1} + \frac{c^2C}{\sigma_1^2}\right) \quad &\text{if  } t\geq \frac{2cC}{\sigma_1}
%		\end{cases}\,.
%	\end{equation}
%	Due to symmetry, we have the same bound for $\Pr\left(\zeta<-t\right)$, so \eqref{eqn:zeta2tail} follows.
%	}
	
  To show the second statement, we notice that
  \begin{equation}\label{eqn:var_var}
    \Ebb\left[ \left(\zeta^2 - \Ebb\left[\zeta^2\right] \right)^2 \right]= \Ebb\left[\zeta^4\right] - \left(\Ebb\left[\zeta^2\right]\right)^2\,.
  \end{equation}
  By considering $\zeta=\sum_{i=1}^p\sigma_i\xi_i\eta_i$, the second
  moment can be calculated directly:
  \begin{equation}\label{eqn:2nd_zeta}
    \Ebb\left[\zeta^2\right] = \Ebb\left[ \sum_{i,j=1}^p \sigma_i \sigma_j  \xi_i \xi_j  \eta_i \eta_j \right] =   \Ebb\left[ \sum_{i=1}^p \sigma_i^2   \xi_i^2 \eta_i^2 \right]
    = \sum_{i=1}^p \sigma_i^2 = 1\,,
  \end{equation}
  where we used the independence of $\xi_i$ and $\eta_i$, the fact
  that $\Ebb\xi_i=\Ebb\eta_i=0$ and $\Ebb \xi_i^2 = \Ebb \eta_i^2 = 1$.

  To control the fourth moment, we
  notice that
  \begin{equation*}
    \Ebb\left[\zeta^4 \right] = \Ebb\left[ \sum_{i,j,k,l}
      \sigma_i\sigma_j\sigma_k\sigma_l \xi_i \xi_j \xi_k \xi_l \eta_i
      \eta_j \eta_k \eta_l \right] \,.
  \end{equation*}
  Due to the independence and the fact that all odd moments vanish for
  Gaussian random variables, the only terms in the summation that
  survive either have all indices equal ($i=j=l=k$) or two indices
  equal to one value while the other two indices equal a different
  value, for example $i=j$ and $k=l$ but $i\neq k$. Altogether, we
  obtain
  \begin{equation*}
    \Ebb\left[\zeta^4  \right] = 3\Ebb\left[ \sum_{i \neq k} \sigma_i^2 \sigma_k^2 \xi_i^2\xi_k^2 \eta_i^2 \eta_k^2\right] + \Ebb\left[ \sum_{i} \sigma_i^4 \xi_i^4 \eta_i^4\right] \,,
  \end{equation*}
  where the coefficient in front of the first term comes from ${4
    \choose 2}/{2 \choose 1} = 3$.  Considering $\Ebb\xi^2=1$ and
  $\Ebb\xi^4 = 3$, we have
  \begin{equation}\label{eqn:4th_zeta}
    \begin{aligned}
      \Ebb\left[\zeta^4  \right] & = 3\sum_{i\neq k} \sigma_i^2 \sigma_k^2 + 9 \sum_{i} \sigma_i^4 =  3\sum_{i,k=1}^p \sigma_i^2 \sigma_k^2 + 6\sum_i \sigma_i^4 \\
      & \leq 3\sum_{i,k=1}^p \sigma_i^2 \sigma_k^2 + 6\sum_i \sigma_i^2 
      = 3\left(\sum_{i=1}^p \sigma_i^2 \right) \left(\sum_{k=1}^p \sigma_k^2 \right) + 6\sum_{i=1}^p \sigma_i^2=9\,,
    \end{aligned}
  \end{equation}
  where we used $\sigma^4_i \leq \sigma^2_i$.  By substituting
  ~\eqref{eqn:2nd_zeta} and~\eqref{eqn:4th_zeta}
  into~\eqref{eqn:var_var}, we have
  \begin{equation*}
    \Ebb\left[ \left(|\zeta|^2 - \Ebb\left[\zeta^2\right] \right)^2 \right] = \Ebb\left[\zeta^4\right] - \left(\Ebb\left[\zeta^2\right]\right)^2 \leq 9 - 1^2 = 8 \,,
  \end{equation*}
  which concludes the proof.
\end{proof}

\blue{
\begin{remark}
We note that this lemma is not new; its proof can be made more compact
if one uses Hanson-Wright inequality and
\cite[Lemma~6.2.2]{vershynin2018high}. The latter result shows that
there exist absolute positive constants $c$ and $C$ such that
\[
\Ebb\left[ \exp(s \xi^\top \Sigma \eta )\right] \leq \exp\left( C s^2 \right)
\]	
for $ |s| \leq \frac{c}{\sigma_1}$. By substituting into
\eqref{eqn:zetatail}, we have
\[
\Pr(\zeta > t) \leq \exp(Cs^2 - st), \;\; \mbox{for all $|s| \leq \frac{c}{\sigma_1}$,}
\]
assuming that the singular value $\sigma_{i}$ on the diagonal of
$\Sigma$ are ordered in a descending manner. Minimizing the right-hand side in
terms of $s$, we have
\begin{equation}
\Pr(\zeta>t) \leq 
\begin{cases}
\exp\left(-\frac{t^2}{4C}\right)\,, \quad &\text{if  } 0\leq t \leq \frac{2cC}{\sigma_1}\\
\exp\left(- \frac{ct}{\sigma_1} + \frac{c^2C}{\sigma_1^2}\right) \quad &\text{if  } t\geq \frac{2cC}{\sigma_1}
\end{cases}\,,
\end{equation}
which, because of symmetry, leads to 
\begin{equation}\label{eqn:zeta3tail}
	\Pr(|\zeta|>t) \leq 
	\begin{cases}
	2\exp\left(-\frac{t^2}{4C}\right)\,, \quad &\text{if  } 0\leq t \leq \frac{2cC}{\sigma_1}\\
	2\exp\left(- \frac{ct}{\sigma_1} + \frac{c^2C}{\sigma_1^2}\right) \quad &\text{if  } t\geq \frac{2cC}{\sigma_1}
	\end{cases}\,.
\end{equation}
This result is rather similar to ours except that the Hanson-Wright
inequality comes with two generic constants $c$ and $C$. These
constants are extremely involved, as shown in the original
proof~\cite{rudelson2013}. We need to make all constants precise, and
thus maintain our full proof with elementary calculations.
%And this would be the same as ours by setting $\sigma_1 \geq
%1/\sqrt{p}$. We would like to keep track of the precise constants and
%thus adopt a more elementary calculation.
\end{remark}
}

\begin{lemma}\label{lem:zetaSampleTail}
  Let $\zeta_i$, $i =1,2,\dotsc,r$ be i.i.d. copies of the random
  variable $\zeta$ defined in Lemma~\ref{lem:zetabound}. Then if
  \begin{equation} \label{eq:rlb}
  r \geq 8 \cdot 3^{3/2} \cdot \max\{ t^{-5/2}\,, p^{3/2} t^{-1}\}\,,
  \end{equation}
  we have 
  \begin{equation}\label{eqn:zetaSampleTail}
    \Pr\left( \left| \frac{1}{r} \sum_{i=1}^r \left( \zeta_i^2 -
    \Ebb\left[\zeta_i^2\right] \right) \right| > t \right) \leq 5r
    \exp \left( \frac34 p^{1/2} \right) \exp\left( -\frac{1}{2} r^{1/3} t^{1/3}\right)\,, \quad
    \mbox{for $t \in [0,1]$}.
  \end{equation}
\end{lemma}
\blue{
\begin{remark}
This lemma essentially deals with the tail bound of a random variable
that is of quartic form of a Gaussian. According to the definition,
$\zeta$ is a quadratic form of Gaussians, and thus is a
sub-exponential, but this lemma considers $\zeta^2$.  Quadratic form
of sub-exponential vectors are studied in \cite{vu2015random}. If we
directly employ their results (especially their Corollary 1.6) by
setting their $\Asf = \frac{1}{r}\Idsf_r \in \Rbb^{r\times r}$, we
obtain, for sufficiently large $r$ (made precise in the corollary)
that
\[
\Pr\left( \left| \frac{1}{r} \sum_{i=1}^r \left(\zeta_i^2 - \Ebb\left[\zeta_i^2\right]\right)\right|>t\right) \leq C \exp\left(-C' \min\left\{ \left(\frac{r^{1/2}t}{\sqrt{\log r}}\right)^{2/3}\,, \left(rt\right)^{1/3} \right\}\right)
\]
where $C$ and $C'$ depend on $p$. We obtain the same power for $r$ and
$t$ as this result, and we make the dependence of the constants on $p$
explicit.
\end{remark}
}
\begin{proof}
Let $E^t$ be the event defined as follows:
  \begin{equation*}
    E^t \defeq  \left\{ \frac{1}{r} \sum_{i=1}^r (\zeta_i^2 - \Ebb\left[\zeta_i^2\right])> t  \right\}\,.
  \end{equation*}
  Due to the symmetry of $\sum_{i=1}^r \zeta_i^2 -
  \Ebb\left[\zeta_i^2\right]$, the probability
  in~\eqref{eqn:zetaSampleTail} is $2\Pr(E^t)$. We now estimate
  $\Pr(E^t)$. For any fixed large number $M$, we define the following
  event, for $i=1,2,\dotsc,r$:
\[
E_i^M \defeq \{ \zeta_i^2 \leq M \}=\{\zeta^2_i-1\leq M-1\}\,.
\]
Clearly, we have
\begin{equation}\label{eqn:split}
  \Pr\left( E^t\right) = \Pr\left(E^t \cap \left( \cap_{i=1}^r E_i^M
  \right)\right) + \Pr\left(E^t \cap \left( \cap_{i=1}^r E_i^M
  \right)^c \right)\,.
\end{equation}
We now estimate the two terms.
\begin{itemize}
\item[1.] For the first term in~\eqref{eqn:split}, we note that
  \begin{equation}\label{eqn:split_cond}
    \Pr \left(E^t \cap \left( \cap_{i=1}^r E_i^M \right) \right) = \Pr
    \left( E^t \ |\ \left( \cap_{i=1}^r E_i^M \right) \right) \cdot
    \Pr \left(\left( \cap_{i=1}^r E_i^M \right) \right) \leq \Pr
    \left( E^t \ |\ \left( \cap_{i=1}^r E_i^M \right) \right) \,.
  \end{equation}
  Denoting $X_i \defeq \zeta_i^2 -\Ebb\left[\zeta_i^2\right]$, and
  realizing that $\Ebb\left[\zeta_i^2\right]=1$ according
  to~\eqref{eqn:Ebbzeta2and4} of Lemma \ref{lem:zetabound}, then $E_i^M = \{ X_i\leq
  M-1\}$. Estimating~\eqref{eqn:split_cond} now amounts to controlling
  the probability of $\sum_{i=1}^rX_i>rt$ assuming that $X_i\leq M-1$
  for all $i=1,2,\dotsc,r$. By applying Bernstein's
  inequality~\eqref{eqn:Bernstein}, we have
  \begin{equation*}
    \begin{aligned}
      \Pr \left( E^t  \ |\ \left( \cap_{i=1}^r E_i^M \right)  \right) &= \Pr\left( \sum_{i=1}^r X_i> rt    \ |\ X_i\leq M-1, \; i=1,2,\dotsc,r \right)\\
      &\leq \exp\left( -\frac{r^2t^2  /2 }{\sum_{i=1}^r\Ebb\left[X_i^2\right] + (M-1)rt/3 }\right)\,.
    \end{aligned}
  \end{equation*}
  From~\eqref{eqn:zetavariance} in Lemma~\ref{lem:zetabound}, we
  have $\Ebb \left[X_i^2\right] \le 8$, so that
  \begin{equation}\label{eqn:split_term1}
    \Pr \left( E^t  \ |\ \left( \cap_{i=1}^r E_i^M \right)  \right) \leq \exp\left( -\frac{3rt^2 }{48 + 2(M-1)t }\right)\,,
  \end{equation}
  which gives the upper bound of the first term in~\eqref{eqn:split}.
\item[2.] For the second term in~\eqref{eqn:split}, we note that
  \begin{equation*}
    \Pr\left(E^t \cap \left( \cap_{i=1}^r E_i^M \right)^c \right) \leq \Pr\left(  \left( \cap_{i=1}^r E_i^M \right)^c \right) = \Pr\left(\cup_{i=1}^r (E_i^M)^c \right) \leq r \Pr\left(  (E_i^M)^c \right)\,.
  \end{equation*}
  By applying \eqref{eqn:zeta2tail} from Lemma \ref{lem:zetabound},
  with $t = \sqrt{M}$, we have
  \begin{equation*}
  \Pr(  (E_i^M)^c ) = \Pr\left( \zeta_i^2 > M \right) = \Pr\left( |\zeta_i| > \sqrt{M}\right) \leq 
  \begin{cases}
  2e^{-\frac{(\sqrt{M}-\sqrt{p})^2}{4\sqrt{p}}} & \;\; \mbox{if $p\leq M \leq 4p$} \\
  2e^{-\frac{(2\sqrt{M}-3\sqrt{p})}{4}} & \;\; \mbox{if $M \geq  4p$} \\
  \end{cases}
  \,,
  \end{equation*}
  and thus
  \begin{equation}\label{eqn:split_term2}
  \Pr(E^t \cap \left( \cap_{i=1}^r E_i^M \right)^c ) \leq 
  \begin{cases}
  2re^{-\frac{(\sqrt{M}-\sqrt{p})^2}{4\sqrt{p}}} & \;\; \mbox{if $p\leq M \leq 4p$} \\
  2re^{-\frac{(2\sqrt{M}-3\sqrt{p})}{4}} & \;\; \mbox{if $M \geq  4p$} \\
  \end{cases}	
  \,.
  \end{equation}
\end{itemize}
By combining~\eqref{eqn:split_term1} and~\eqref{eqn:split_term2}
in~\eqref{eqn:split}, we have
\begin{equation} \label{eq:tj0}
\Pr(E^t ) \leq \exp\left( -\frac{3rt^2 }{48+ 2(M-1)t} \right)+
\begin{cases}
2re^{-\frac{(\sqrt{M}-\sqrt{p})^2}{4\sqrt{p}}} & \;\; \mbox{if  $p\leq M \leq 4p$} \\
2re^{-\frac{(2\sqrt{M}-3\sqrt{p})}{4}} & \;\; \mbox{if $M \geq  4p$} \\
\end{cases}\,.
\end{equation}

To find a sharp bound of $\Pr(E^t)$, we choose a suitable value of
$M$. We set 
\begin{equation} \label{eq:def.M}
  M =  r^{2/3} t^{2/3},
\end{equation}
where $r$ satisfies the lower bound \eqref{eq:rlb}. Since $r \ge 8
\cdot 3^{3/2} \cdot p^{3/2} t^{-1}$, we have $r^{2/3} \ge 12 p
t^{-2/3}$, so that
\begin{equation} \label{eq:tj1}
  M = r^{2/3}  t^{2/3} \ge 12p > 4p,
\end{equation}
so the second case applies in \eqref{eq:tj0}. Since $r \ge 3^{3/2}
\cdot 2^3 \cdot t^{-5/2}$, we have $r^{2/3} \ge 12
t^{-5/3}$, so that
\[
Mt = r^{2/3} t^{5/3} \ge 12,
\]
so that, for the denominator of the first term in \eqref{eq:tj0}, we
have
\begin{equation} \label{eq:tj2}
48+2(M-1)t = 6Mt + 48 - 2t - 4Mt \le 6Mt.
\end{equation}
By using these observations in \eqref{eq:tj0}, we have for the value
\eqref{eq:def.M} that
\begin{equation} \label{eq:tj3}
  \Pr(E^t ) \leq \exp \left( -\frac12 \frac{rt}{M} \right) +
  2r \exp \left( \frac34 p^{1/2} \right) \exp \left( -\frac12 M^{1/2} \right).
\end{equation}
With $M$ defined as in \eqref{eq:def.M}, we see that the two
exponential terms involving $M$ in this expression are both equal to
$\exp (-r^{1/3} t^{1/3}/2)$. Additionally, since $p \ge 1$ and $r \geq
1$, we have $2r \exp(3p^{1/2}/4) >4$. Thus, from \eqref{eq:tj3}, we
obtain
\begin{equation} \label{eq:tj4}
  \Pr(E^t ) \leq  (5/2) r \exp \left( \frac34 p^{1/2} \right) \exp \left( -\frac12 r^{1/3} t^{1/3} \right).
\end{equation}
We obtain the result by multiplying the right-hand side by $2$, as
discussed at the start of the proof.
%% Recalling that $t \le 1$, we choose $M$ large enough that $M\geq
%% 4p$ and $(M-1)t \geq 1$, to have:
%% \begin{equation}\label{eqn:tailM}
%%   \Pr(E^t) \leq 2re^{3\sqrt{p}/4} \exp\left( -\frac{\sqrt{M}}{2}\right) + \exp\left(-\frac{3 rt^2}{50(M-1)t}\right) \,, \quad \forall 0\leq t\leq 1\,.
%% \end{equation}
%% Further set $M$ so that: \sw{Oh man it's really sloppy from here on in
%%   this proof.}
%% \[
%% \frac{\sqrt{M}}{2} = \frac{3rt^2}{50(M-1)t} \implies M = \tilde{C_1} r^{2/3} t^{2/3} \quad \text{for some constant }\tilde{C_1}>0\,,
%% \]
%% We combine $M = \tilde{C_1} r^{2/3}$ with the restrictions $(M-1)t\geq 1$ and $M\geq 4p$, and see that
%% \[
%% r \geq C_1 \max\{ t^{-5/2}\,, p^{3/2} t^{-1}\}\quad \text{for some constant } C_1>0\,,
%% \]
%% then the estimate~\eqref{eqn:tailM} gets tightened to
%% \[
%% \Pr(E^t) \leq (2r e^{3\sqrt{p}/4}+1) \exp\left( -C_2 r^{1/3} t^{1/3}\right) \,,
%% \]
%% for some constant $C_2>0$ and all $0\leq t \leq 1$. It finishes the proof for the lemma.
\end{proof}

Proposition~\ref{prop:one_y} is a direct consequence of
Lemmas~\ref{lem:SyDistribution} and \ref{lem:zetaSampleTail}.

\begin{proof}[Proof of Proposition~\ref{prop:one_y}]
For any $\ysf\in\text{Span}\{\Fsf\otimes\Gsf\}$, denote
$\hat{\ysf}=\frac{\ysf}{\|\ysf\|}$, so that $\|\hat{\ysf}\|=1$. From
Lemma~\ref{lem:SyDistribution}, we have
\begin{equation*}
  \| \Ssf \hat\ysf\|^2 \dsim \frac{1}{r}\sum_{i=1}^r \zeta_i^2 \,, \quad \text{where } \zeta_i \defeq \xi_i^\top \Sigma \eta_i\,,
\end{equation*}
where $\xi_i,\eta_i \in \Rbb^p$ are independent Gaussian
vectors drawn from $\mathcal{N}(0,\Idsf_p)$. We have 
\[
\frac{\| \Ssf \ysf\|^2 - \|\ysf\|^2}{\|\ysf\|^2}=\frac{\| \Ssf \hat\ysf\|^2 - \|\hat\ysf\|^2}{\|\hat\ysf\|^2} = \frac{1}{r}\sum_{i=1}^r \zeta_i^2-1\,.
\]
By setting $t = \vep$ in \eqref{eqn:zetaSampleTail} from
Lemma~\ref{lem:zetaSampleTail}, we have
\[
\Pr\left(\left|\frac{\| \Ssf \ysf\|^2 -
  \|\ysf\|^2}{\|\ysf\|^2}\right|>\vep\right)=\Pr\left( \left|
\frac{1}{r} \sum_{i=1}^r (\zeta_i^2 - 1) \right| > \vep \right) \leq
5r \exp \left(\frac34 p^{1/2} \right) \exp\left( -\frac12 r^{1/3}
\vep^{1/3}\right)\,,
\]
conditioned on $r \geq 8 \cdot 3^{3/2} \cdot \max\{ \vep^{-5/2},
p^{3/2} \vep^{-1}\}$, as required.
\end{proof}

\subsection{Proof of Theorem~\ref{thm:maincase2}}\label{sec:case2_net}

Proposition~\ref{prop:one_y} shows the probability of the sketching
matrix $\Ssf$ of the form \eqref{eqn:formScase2} preserving the norm
of a fixed given vector in the range space $\Range{\Fsf \otimes
  \Gsf}$. To show the preservation of norm holds true over the entire
column space, we follow the construction
of~\cite{woodruff2014sketching}.  We construct a $\gamma$-net over the
unit sphere in $\Range{\Fsf\otimes\Gsf}$ and show that for $r$
sufficiently large, with high probability, the angles between any
vectors in the net will be preserved with high accuracy. Preservation
of angles on the $\gamma$-net can be translated to the norm
preservation over the entire space.

We show in Lemma~\ref{lem:JLT} that angles can be preserved with the
sampling matrix $\Ssf$ of the form \eqref{eqn:formScase2}.  In
Lemma~\ref{lem:cardNcal}, we calculate the cardinality of the
$\gamma$-net. The fact that preservation of angle leads to the
preservation of norms on the space is justified in
Lemma~\ref{lem:angle_norm}. The three results can be combined into a
proof for Theorem~\ref{thm:maincase2}, which we complete at the end of
the section.

\begin{lemma}\label{lem:JLT}
  Let $V$ be a collection of vectors in $\Rbb^n$ with cardinality
  $|V|=f$ and let
  \[
  \tilde{V} \defeq \{ \usf \pm \vsf: \usf,\vsf \in V  \}\,.
  \]
  Suppose that a random matrix $\Ssf$ preserved norm on $V$, in the
  sense that for each $\tilde{\vsf}\in\tilde{V}$, with probability at
  least $1-\delta$, we have 
\[
\left| \|\Ssf\tilde\vsf\|^2 - \|\tilde{\vsf}\|^2\right| < \vep \|\tilde\vsf\|^2\,.
\]
%%   $(\vep,\delta)$-$l^2$ preserves norm
%%   of all $\tilde{\vsf}\in\tilde{V}$. That is, for each
%%   $\tilde{\vsf}\in\tilde{V}$, with probability at least $1-\delta$, if we
%%   have
%% \[
%% \left| \|\Ssf\tilde\vsf\|^2 - \|\tilde{\vsf}\|^2\right| < \vep \|\tilde\vsf\|^2\,, 
%% \]
Then $\Ssf$ preserves the angle between all elements in $V$ with
probability at least $1-4f^2 \delta$, that is, 
\[
\Pr\left( \left|\langle \Ssf \usf,\Ssf \vsf \rangle - \langle \usf
,\vsf \rangle \right| \leq \vep \| \usf \| \| \vsf
\|\right)>1-4f^2\delta\,, \quad \mbox{for all $\usf,\vsf \in V$}.
\]
\end{lemma}
\begin{proof}
Without loss of generality, we assume all vectors in $V$ are unit
vectors. Because of the assumptions on $\Ssf$, we have
%% It is straightforward to see that $|\tilde{V}| \leq f^2$. Since
%% $\Ssf$ $(\vep,\delta)$-$l^2$ embeds all vectors in $\tilde{V}$, we
%% have
\begin{equation}\label{eqn:tilde_V_prob}
\Pr\left( \left| \|\Ssf\tilde\vsf\|^2 - \|\tilde{\vsf}\|^2\right| <
\vep \|\tilde\vsf\|^2  \;\; \mbox{for all 
$\tilde{v}\in\tilde{V}$} \right) \leq 1-f^2\delta.
\end{equation}

Considering $\usf,\vsf\in V$, we denote $\ssf \defeq \usf + \vsf \in
\tilde{V}$ and $\tsf \defeq \usf - \vsf\in \tilde{V}$ and use the
parallelogram equality:
\begin{equation*}
\langle \usf,\vsf\rangle = \frac{1}{4}\left( \| \ssf \|^2 - \|\tsf \|^2 \right)\,,\quad\langle \Ssf \usf,\Ssf \vsf\rangle = \frac{1}{4}\left( \| \Ssf \ssf \|^2 - \|\Ssf \tsf\|^2 \right) \,,
\end{equation*}
so that
\[
\langle \Ssf \usf,\Ssf \vsf\rangle - \langle \usf,\vsf\rangle = \frac{1}{4} \left( \|\Ssf \ssf \|^2 - \|\ssf \|^2 - (\|\Ssf \tsf\|^2 - \|\tsf\|^2)\right)\,.
\]
From~\eqref{eqn:tilde_V_prob}, we have, with probability at least
$1-f^2\delta$, \blue{for all $\usf,\vsf\in V$}
\begin{equation*}
\begin{aligned}
\left|\langle \Ssf \usf,\Ssf \vsf\rangle - \langle \usf,\vsf\rangle \right| &\leq \frac{1}{4} \left( \left|\|\Ssf \ssf \|^2 - \|\ssf \|^2\right| +\left| \|\Ssf \tsf\|^2 - \|\tsf\|^2\right|\right)\\
&\leq \frac{\vep}{4}( \| \ssf\|^2+\| \tsf\|^2 ) =  \frac{\vep}{4}( \| \usf + \vsf \|^2 + \| \usf -\vsf\|^2) = \frac{\vep}{4}(2 \|\usf \|^2 + 2\|\vsf \|^2)=\vep\,,
\end{aligned}
\end{equation*}
which completes the proof.
\end{proof}

%% We now follow the standard $\gamma$-net argument in sketching
%% theory~\cite{woodruff2014sketching} to prove
%% Theorem~\ref{thm:maincase2}.

We now define the $\gamma$-net, and show that preservation of angles
on this net leads to preservation of norms.

\begin{definition} \label{def:SN}
Denote the unit sphere in space $\Range{\Fsf\otimes\Gsf}$ by $\Scal$, that is,
\begin{equation}\label{def:Space}
\Scal \defeq \left\{ \ysf \in \Rbb^{n_1 n_2}:\  \ysf= (\Fsf \otimes \Gsf)\xsf \text{ for some }\xsf \in \Rbb^{p^2} \text{ and } \| \ysf \| = 1 \right\}\,.
\end{equation}
For fixed $\gamma \in (0,1)$, we call $\Gcal$ a $\gamma$-net of
$\Scal$ if $\Gcal$ is a finite subset of $\Scal$ such that for any
$\ysf\in \Scal$, there exists $\wsf \in \Gcal$ such that $ \| \wsf -
\ysf \| \leq \gamma$.
\end{definition}
	
\blue{The following lemma was presented in
  \cite[Section~2.1]{woodruff2014sketching}.}
% {We include its proof for completeness.}
% \footnote{SJW: I have to check the referee reports and the response
% on this point, but this looks like a weak reason to keep the proof.}

\begin{lemma}\label{lem:angle_norm}
  Let $\Scal$ and $\Gcal$ be as in Definition~\ref{def:SN}, for some
  $\gamma \in (0,1)$. Then preservation of angle on $\Gcal$ leads to
  the preservation of norm in $\Scal$. That is, if 
  \begin{equation} \label{eq:ex9}
    \left| \langle \Ssf \wsf,\Ssf \wsf' \rangle - \langle  \wsf ,\wsf' \rangle \right| \leq \vep \,, \quad \mbox{for all $\wsf$, $\wsf \in \Gcal$,}
  \end{equation}
  then
  \begin{equation*}
    % \label{eqn:embedNcal}
    \left| \| \Ssf \ysf \|^2 - \| \ysf \|^2 \right| \leq  \frac{\vep}{(1-\gamma)^2} \,, \quad \mbox{for all $\ysf \in \Scal$.}
  \end{equation*}
\end{lemma}

The size of the $\gamma$-net can also be controlled, as we now show.
%\blue{Versions of the lemma below with tighter constants can also be
%  found in \cite[Lemma~2.2]{woodruff2014sketching} and
%  \cite[Corollary~4.2.13]{vershynin2018high}. We repeat it here for
%  the completeness.}
%\footnote{SJW: Does ``tighter constants'' mean thay they have a
%  stronger result? If so, shouldn't we use their result? And leave out
%  the proof? This is a hard business. Even if we had the idea
%  independently, I don't think we can publish the same or a weaker
%  result later than one already in print. ~\kc{\cite{woodruff2014sketching} actually put some of the proof here into paragraph of disucssions. \cite{vershynin2018high} has a smaller constant like $(1+\frac{2}{\gamma})^{p^2}$. I think it is okay to just delete the proof and cite them. }\ql{let us delete our proof and cite Vershynin paper and use $2$ throughout the paper. It won't change too much of the proof -- just a few numbers.}}

\blue{\begin{lemma}\label{lem:cardNcal}
  Let $\Scal$ the the unit sphere of $\Fsf\otimes\Gsf$, defined
  in~\eqref{def:Space}. Then for any $\gamma \in (0,1)$, there exists
  a $\gamma$-net $\Gcal$ of $\Scal$ such that
  \[
  | \Gcal | \leq \left( 1+ \frac{2}{\gamma} \right)^{p^2}\,.
  \]
\end{lemma}
\begin{proof}
	Notice that $\Scal$ is isometric to the unit Euclidean sphere $\Scal^{p^2-1}$, the result follows directly by applying Corollary 4.2.13 of ~\cite{vershynin2018high}.
\end{proof}
}

  Finally, we state the proof of Theorem~\ref{thm:maincase2}, which is
  obtained from the lemmas in this section together with
  Proposition~\ref{prop:one_y}.
  
\begin{proof}[Proof of Theorem~\ref{thm:maincase2}]
  Without loss of generality, it suffices to show $\Ssf$ preserves
  norm with high accuracy and high probability over the unit sphere in
  $\Range{\Fsf \otimes \Gsf}$, defined by
  \[
  \Scal \defeq \{ \ysf \in \Rbb^{n_1 n_2}:\  \ysf= (\Fsf \otimes \Gsf)\xsf \text{ for some }\xsf \in \Rbb^{p^2} \text{ and } \| \ysf \| = 1 \}\,.
  \]
  Note from Lemma~\ref{lem:cardNcal} that for given $\gamma \in
  (0,1)$, one can construct a $\gamma$-net $\Gcal$ of $\Scal$ of size
  $f=(1+ \frac{2}{\gamma})^{p^2}$. Given $\vep_1 \in (0,1/2)$, then on
  this $\Gcal$, according to Proposition~\ref{prop:one_y} and
  Lemma~\ref{lem:JLT}, if we assume
  \begin{equation}\label{eqn:rLow1}
    r \geq 8 \cdot 3^{3/2} \cdot  \max\{ \vep_1^{-5/2}\,, p^{3/2} \vep_1^{-1}\}
  \end{equation}
  then with probability at least $1-\delta_2$ with
  \begin{equation}\label{eqn:delta_case2}
    \delta_2 \le  20 r f^2 \exp \left( \frac34 p^{1/2} \right) \exp \left(
    -\frac12 r^{1/3} \vep_1^{1/3} \right) =
     20 r \left( 1 + \frac{2}{\gamma} \right)^{2p^2} \exp \left( \frac34 p^{1/2} \right) \exp \left(
    -\frac12 r^{1/3} \vep_1^{1/3} \right),
  \end{equation}
we have that $\Ssf$ preserves angles, that is,
  \[
  \left| \langle \Ssf \wsf,\Ssf \wsf' \rangle - \langle  \wsf ,\wsf' \rangle \right| \leq \vep_1 \,, \quad \text{for all }\quad \wsf,\wsf' \in \Gcal \,,
  \]
  According to Lemma~\ref{lem:angle_norm}, $\Ssf$ embeds $\Scal$, that
  is,
  \[
  \left| \| \Ssf \ysf \|^2 - \| \ysf \|^2 \right| \leq
  \vep \,, \quad \text{for all } \ysf \in
  \Scal, \quad \mbox{where} \;\; \vep  \defeq \frac{\vep_1}{(1-\gamma)^2}.
  \]

  First, we need to convert the condition \eqref{eqn:rLow1} into one
  involving $\vep$. We obtain
  \begin{equation} \label{eq:rlo5}
  r   \geq 8 \cdot 3^{3/2} \cdot  \max\{ \vep^{-5/2} (1-\gamma)^{-5}, p^{3/2} \vep^{-1}(1-\gamma)^{-2}\}.
  \end{equation}
Second, we must alter the lower bound on $r$ to ensure that the
right-hand side of \eqref{eqn:delta_case2} is smaller than the given
value of $\delta$, that is,
  \begin{equation} \label{eq:us7}
    \delta \ge 20 r \left( 1 + \frac{2}{\gamma} \right)^{2p^2} \exp
    \left( \frac34 p^{1/2} \right) \exp \left( -\frac12 r^{1/3}
    \vep^{1/3} (1-\gamma)^{2/3} \right),
  \end{equation}
  or equivalently,
  \begin{equation} \label{eq:us8}
    \log \delta \ge \log 20 + \log r + 2p^2 \log(1+2/\gamma) +
    \frac34 p^{1/2} - \frac12 r^{1/3} \vep^{1/3} (1-\gamma)^{2/3}.
    \end{equation}
  Note that for $p\ge 6$ and $\gamma \in (0,1)$, we have $\log 20 <
  3<.1 p^2 \log(1+2/\gamma)$ and $.75p^{1/2} < .1 p^2
  \log(1+2/\gamma)$. Thus a sufficient condition for \eqref{eq:us8} is
  \begin{equation} \label{eq:us9}
    \log \delta \ge \log r + 2.2 p^2 \log(1+2/\gamma) - \frac12
    r^{1/3} \vep^{1/3} (1-\gamma)^{2/3}.
  \end{equation}
 Denoting
 \[
 \alpha \defeq \vep^{1/3} (1-\gamma)^{2/3}\quad \text{and } \quad\beta \defeq \frac{1}{3} \left(2.2 p^2 \log(1+2/\gamma) + \left| \log \delta \right| \right) \,,
 \]
 we have $\alpha \in (0,1)$ for any $\vep,\gamma \in (0,1)$. By using
 these definitions, we see that \eqref{eq:us9} is equivalent to
 \begin{equation}\label{eq:us10}
 	\frac{\alpha}{6} r^{1/3} - \log r^{1/3} \geq \beta,
 \end{equation}
 for which the combination of the following two conditions is
 sufficient:
 \begin{subequations}
 \begin{align}
\frac{\alpha}{12} r^{1/3} - \log r^{1/3} &  \geq 0\,, \label{eq:us11_1}\\
 \frac{\alpha}{12} r^{1/3}& \geq \beta\,. \label{eq:us11_2}
 \end{align}
 \end{subequations}
Condition \eqref{eq:us11_2} can be rewritten to
\[
r \geq \frac{12^3\beta^3}{\alpha^3} =\frac{4^3}{\vep (1-\gamma)^2} \left(2.2 p^2 \log(1+2/\gamma) + \left| \log \delta \right| \right)^3 \,,
\]
for which a sufficient condition is
\begin{equation}\label{eqn:second_cond}
r \geq \frac{8.8^3}{\vep (1-\gamma)^2} \log^3(1+2/\gamma)\left(  p^2 + \left| \log \delta \right| \right)^3\,.
\end{equation}
The condition~\eqref{eq:us11_1} requires $h(r^{1/3})\geq 0$, where
$h(x) \defeq \frac{\alpha}{12} x - \log x$. Since
\[
h'(x) = \frac{\alpha}{12}-\frac{1}{x} \geq 0\,,
\]
we see that $h$ is an increasing function for $x> 12/\alpha$. By
noting that
\[
h\left(\frac{12}{\alpha^{5/2}}\right) = \alpha^{-3/2} - \log(12) + \frac{5}{2} \log \alpha \geq 0, \quad \mbox{for $\alpha \in (0,0.33)$}\,,
\]
and
\[
\frac{12}{\alpha^{5/2}}>\frac{12}{\alpha}\,,\quad\alpha\in(0,1)\,,
\]
we have for $\alpha\in(0,0.33)$ that 
\[
h(r^{1/3}) \geq0\,,\quad \mbox{if} \;\; r^{1/3} \geq \frac{12}{\alpha^{5/2}}\,,
\]
which leads to
\begin{equation}\label{eqn:first_cond}
r\geq \frac{12^3}{\vep^{5/2}(1-\gamma)^5}\,.
\end{equation}
We are free to choose $\gamma \in (0,1)$ in a way that ensures that
$\alpha \in (0,.33)$. In fact, by setting $\gamma=3/4$, we have
\[
\alpha = \vep^{1/3}(1/4)^{2/3} < 0.33\,, \quad \mbox{for all
  $\vep \in (0,0.5)$}.
\]

By combining the conditions \eqref{eqn:first_cond}
and~\eqref{eqn:second_cond}, and setting $\gamma=3/4$, we have
\[
r \geq \max\left\{ \frac{\bar C_1}{\vep^{5/2}}\,, \frac{\bar C_2}{\vep } \left(  p^2 + \left| \log \delta \right| \right)^3 \right\}\,,
\]
with $\bar C_2 = 8.8^3 \cdot 4^2 \log^3(11/3) \approx 2.4e4$, and $\bar C_1 =12^3\cdot 4^5$.
 \end{proof}

We could change the weight in the separation of \eqref{eq:us10} into
\eqref{eq:us11_1} and \eqref{eq:us11_2}, one could arrive at different
(possibly better) constants $\bar C_1$ and $\bar{C}_2$ in the final
expression. However, our priority is to show dependence of $r$ on
$\vep$, $\delta$, and $p$ (and {\em not} $n$), and optimization of the
constants is less important.

% ZZZZZZZZZZZ

\section{Numerical Tests}\label{sec:numerical}

\blue{This section presents some numerical evidence of the
  effectiveness of our sketching strategies. We test them on general
  matrices with the tensor structure and a problem directly from
  EIT~\eqref{eqn:inverse_eit}. We are mostly concerned of the
  dependence of accuracy on $n$, $r$, and $p$. The computational
  complexity is rather straightforward and is omitted from
  discussion. In both tests, the numerical solutions outperform the
  theoretical predictions, indicating that there is room for
  improvement in our bounds for $r$.}

\subsection{General matrices with tensor structure}
To set up the experiment, we generate two
matrices $\Fmat= [\fmat_1,\ldots,\fmat_p]\in \Rbb^{n_1\times p}$ and
$\Gmat=[\gmat_1,\ldots,\gmat_p]\in\Rbb^{n_2\times p}$ using:
\[
\Fsf = \Usf_\Fsf \Sigma_\Fsf \Vsf_\Fsf^\top\, \;\; \text{and} \;\; \Gsf = \Usf_\Gsf \Sigma_\Gsf \Vsf_\Gsf\,,  
\]
where $\Usf_\Fsf \in \Rbb^{n_1 \times p}$, $\Usf_\Gsf\in
\Rbb^{n_2\times p}$, $\Vsf_\Fsf \in \Rbb^{p\times p}$, and
$\Vsf_\Gsf\in \Rbb^{p\times p}$ are generated by taking the
QR-decomposition of random matrices with i.i.d Gaussian entries. The
diagonal entries of $\Sigma_\Fsf$ and $\Sigma_\Gsf$ are independently
drawn from $\Ncal(1,0.04)$. Matrix $\Asf\in\Rbb^{n\times p}$ is then
defined by setting $\amat_j=\fmat_j\otimes\gmat_j$, with
$n=n_1n_2$. We further generate the reference solution
$\xmat_\text{ref}\in \Rbb^p$ whose entries are drawn from
$\mathcal{N}(1,0.25)$. The right-hand-side vector $\bmat\in\Rbb^n$
encodes a small amount of noise; we set
\[
\bmat = \Amat \xmat_\text{ref} + 10^{-6}\xi\,.
\]
where each entry of $\xi$ is drawn from $\mathcal{N}(0,1)$. We compute
$\xmat^\ast$ using~\eqref{eqn:LS}.

Three sketching strategies will be considered, the first two cases
from \eqref{eqn:case1} and \eqref{eqn:case2}, and a third standard
strategy that does not take account of the tensor structure in
$\Amat$.
\begin{description}
\item[Case 1] Set $\Smat =\Pmat\otimes\Qmat$ (normalized), as defined
  in~\eqref{eqn:case1} with entries in $\Pmat\in\Rbb^{r_1\times n_1}$
  and $\Qmat\in\Rbb^{r_2\times n_2}$ drawn i.i.d. from
  $\mathcal{N}(0,1)$. Notice here that $r = r_1r_2$.
\item[Case 2] Set $\Smat_{i,:} =\pmat_i^\top\otimes\qmat_i^\top$
  (normalized), as defined in~\eqref{eqn:case2}, with entries in
  vectors $\{\pmat_i\}$ and $\{\qmat_i\}$ drawn i.i.d.  from
  $\mathcal{N}(0,1)$ for all $i=1,\ldots,r$.
\item[Random Gaussian] $\Smat = \Rmat \in\Rbb^{r\times n}$
  (normalized), with entries in $\Rmat$ drawn i.i.d. from
  $\mathcal{N}(0,1)$.
\end{description}

The random Gaussian choice is not practical in this context, but
% is widely used, with well understood properties. We
we include it here as a reference.

%% The standard sketching strategy does not impose any tensor
%% structure and is therefore not practical. It is also bound to
%% attain a better estimate in general. Despite its impracticality, we
%% still include it here as a reference.

For these three choices of $\Smat$, we compute the solution
$\xmat^\ast_s$ of the sketched LS problem \eqref{eqn:LS_ast}, and
compare the sketching solution with the standard least-squares
solution. In particular, we evaluate the following relative error
\begin{equation}\label{eqn:relerrResidue}
\text{Error} = \frac{f(\xmat^\ast_s) - f(\xmat^\ast)}{f(\xmat^\ast)}\,,\quad\text{with}\quad f(\xmat)=\|\Amat\xmat-\bmat\|_2^2\,.
\end{equation}
%% Since the value of the relative error depends on the particular
%% drawing of $\Smat$, in experiment, with each strategy, w
For each strategy, we draw $10$ independent samples of $\Smat$ and
compute the median relative error. We discuss how this quantity
depends on $r$ and $n$.
% ~\kc{changed from 5 trials to 10 trials and from using average to median.}
	
\subsubsection*{\textbf{Dependence  on $r$}} 
We set $\vep=0.5$, $\delta = 10^{-3}$, $p=10$, and $n_1=n_2=10^2$, and
choose the following values for $r$: $256$, $1024$, $4096$, $16384$
and $65536$. As shown in Figure~\ref{fig:rTest}, the relative error
for all three strategies decreases as $r$ increases; all are of the
order of $r^{-1}$. The result suggests Case-2 sketching and the
Gaussian reference sketching share almost the same accuracy, while
Case 1 is slightly worse.
% We note that the random Gaussian reference is not computationally
% feasible in practice since it does not consider the tensor structure
% in $\mathsf{S}$.

\begin{figure}[htp]
\centering
\includegraphics[width = 0.6\textwidth]{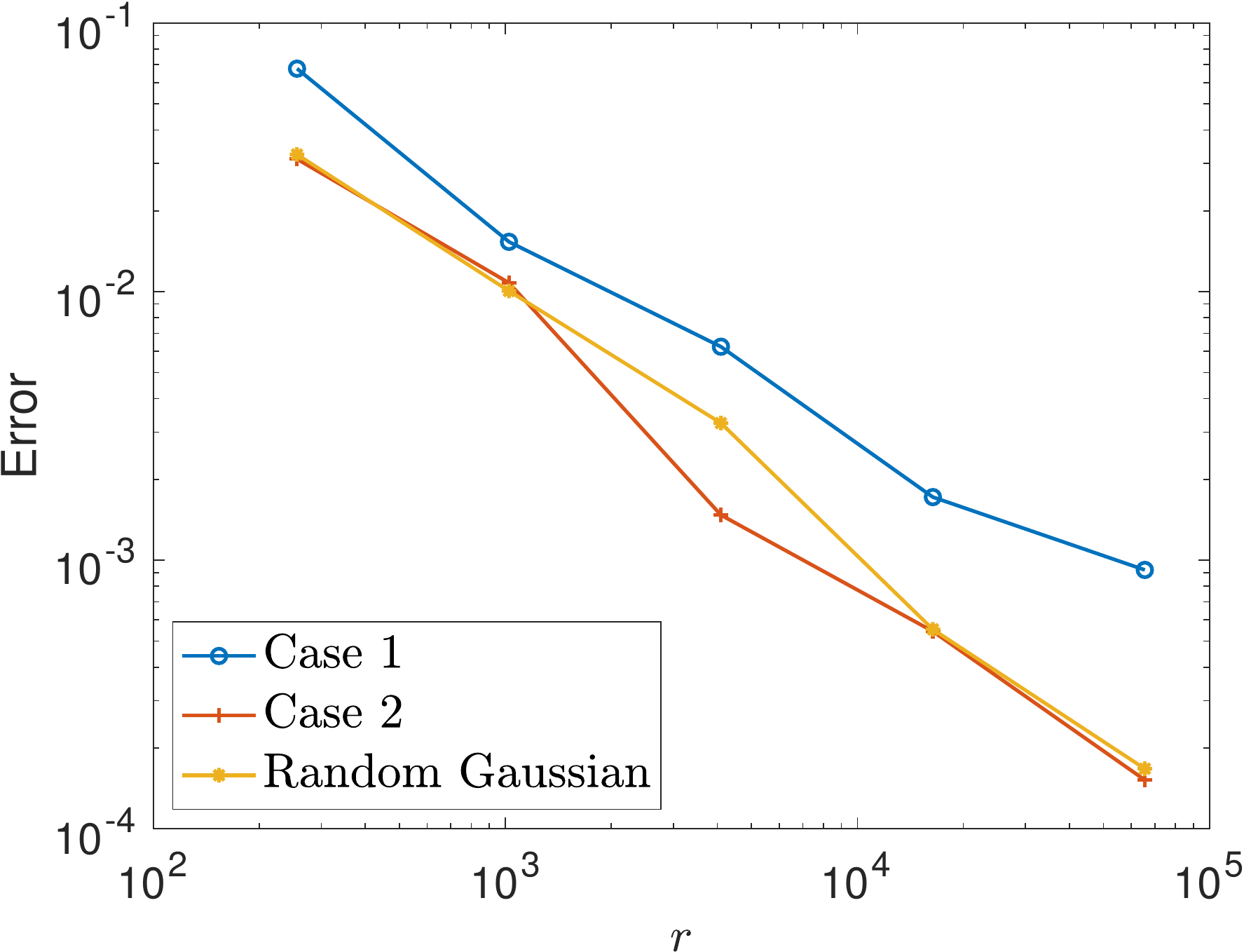}
\caption{Dependence of relative error on $r$ for the three sketching
  strategies.\label{fig:rTest}}
\end{figure}

%\subsubsection*{Dependence on $\vep$}
%We set $p$, $n$, and $\delta$ as above, and choose $r$ as
%in~\eqref{eqn:r_bound1} in Theorem~\ref{thm:maincase1}, that is,
%\[
%r = r_1 r_2\quad \text{with}\quad r_1 = r_2 =\frac{1}{\vep^2}(|\log \delta| + p)\,.
%\]
%Table~\ref{tbl:r} shows the dependence of $r$ on $\vep$, for the
%values of $\vep$ of interest to us. 
%
%\begin{table}[htp]
%\centering
%\begin{tabular}{|c|c|c|c|c|c|}
%\hline
%$\vep$ & 0.9 & 0.8 & 0.7 & 0.6 & 0.5 \\
%\hline
%$r$ & 441 & 729 & 	1225 & 2209 & 4624 \\
%\hline
%\end{tabular}
%\caption{Dependence of the number of rows $r$ in the sketching matrix $\Smat$ on the parameter  $\vep$.}\label{tbl:r}
%\end{table}
%
%In Figure~\ref{fig:vepTest}, we plot the dependence of the relative
%error on $\vep$, averaging over five tests as above.  The random
%Gaussian strategy generally gives the best relative error, but the
%Case 2 strategy, which is much more practical, is nearly optimal,
%particularly for larger values of $\vep$.  The Case 1 strategy is also
%competitive, for values of $\vep$ closer to $1$.
%
%\begin{figure}
%\centering \includegraphics[width=0.6\textwidth]{test_on_vep_new.pdf}
%\caption{Dependence of relative error on $\vep$ for the three
%  sketching strategies.\label{fig:vepTest}}
%\end{figure}

\subsubsection*{\textbf{Dependence on $n$}}
Theorems~\ref{thm:maincase1} and \ref{thm:maincase2} suggest
essentially no dependence on $n$. To test this claim empirically, we
fix $\vep = 0.5$, $\delta = 10^{-3}$, and $r = 2209$, and set
$n_1=n_2$ to be $50$, $100$, $150$, $200$, $250$. The error, plotted
in Figure~\ref{fig:mTest}, shows no dependence on $n$.

\begin{figure}[htp]
\centering
\includegraphics[width=0.6\textwidth]{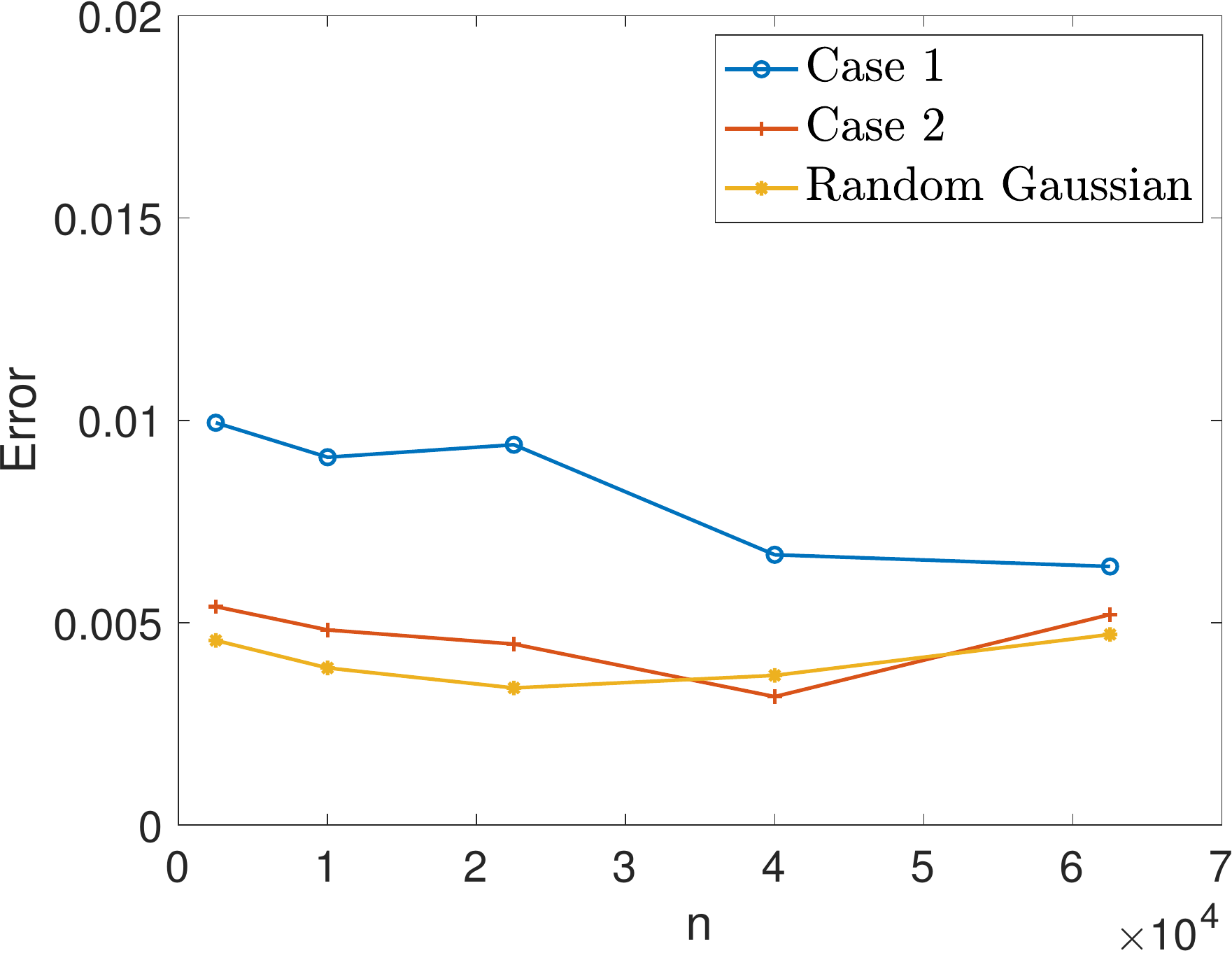}
\caption{Dependence of relative error on ambient dimension $n$ for the
  three sketching strategies.\label{fig:mTest}}
\end{figure}

\subsubsection*{\textbf{Dependence on $p$}}
In this experiment, we study the dependence of relative error on
$p$. We fix $\vep = 0.5$, $\delta = 10^{-3}$, and $r = 4096$ and let
$p$ take the values $3$, $6$, $9$, $12$, $15$. The results are plotted
in Figure~\ref{fig:pTest}. The plot seems to indicate linear
dependence on $p$, better than the higher powers of $p$ predicted by
our bounds.  %This better performance suggests that our bounds could be
%tightened, or it may be due to the fact that $\Fmat$ and $\Gmat$ in
%this experiment are well-conditioned.% \footnote{SJW: I think referees will complain if we leave it like this and don't try to TEST whether well-conditioning of $F$ and $G$ is responsible. \kc{do we want to test for ill-conditioned $F$ and $G$? Alternatively, as the EIT example is ill-posed, do you think adding a $p$-test for EIT is good enough?}\ql{we need to discuss this.}~\kc{Why do we mention it in the new $p$-test here? I think the referee was saying that in the $r$ and $n$ tests, the case 2 aligned with Gaussian might be a consequence of nice condition number of $A$. Now in our $r$-test of EIT, we see that case 2 also aligned with Gaussian, will that be enough to convince the referee?}}
We leave the discussion to future research.

\begin{figure}[htp]
	\centering
	\includegraphics[width=0.6\textwidth]{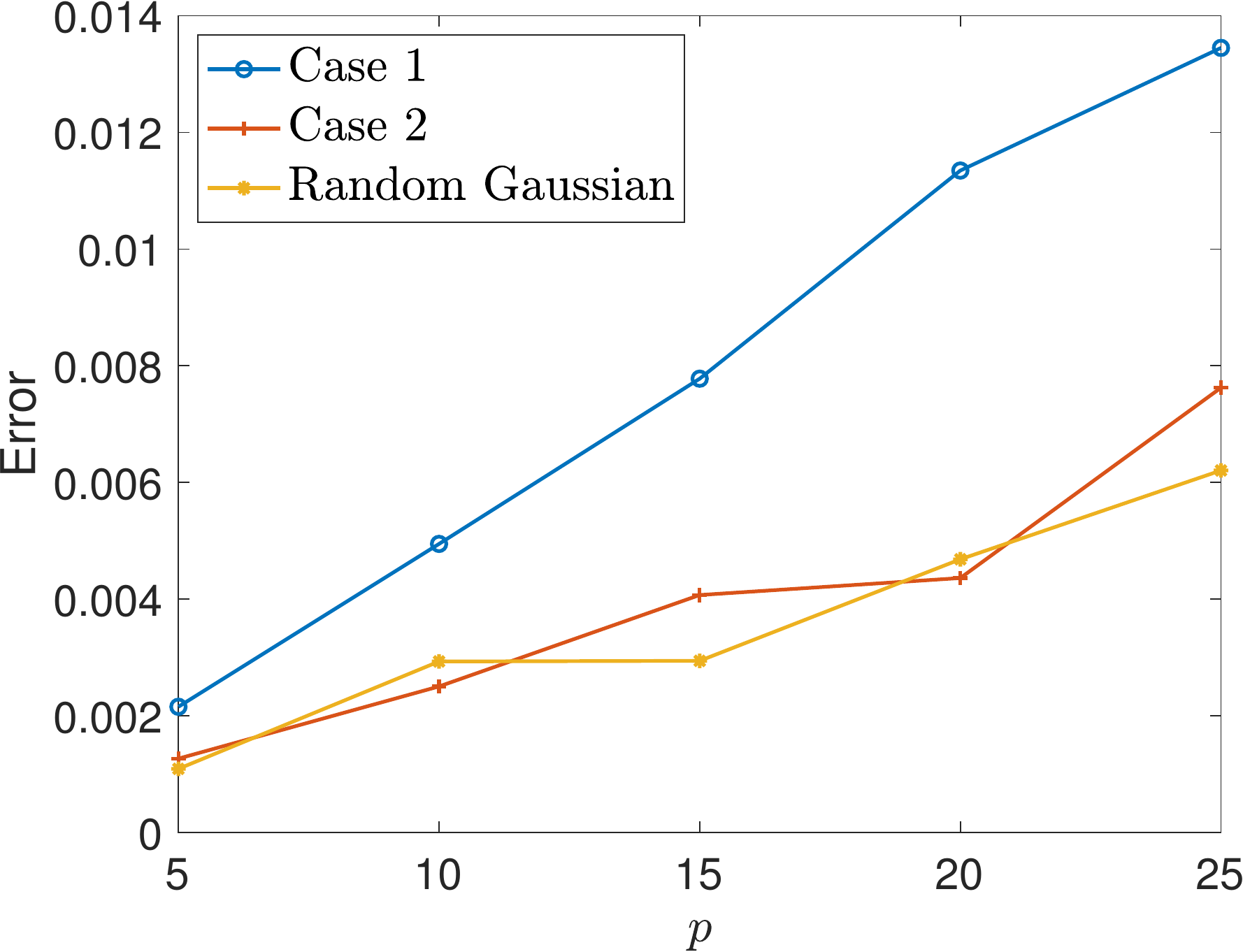}
	\caption{Dependence of relative error on number of unknowns  $p$ for the
		three sketching strategies.\label{fig:pTest}}
\end{figure}

\subsection{Electrical Impedance Tomography}
\blue{In this section, we study the EIT inverse problem on a unit
  square $[0,1]^2$. As presented in Section~\ref{sec:inverse}, the
  goal is to reconstruct the conductivity function $\sigma(x)$ in
  \eqref{eqn:inverse_eit}. We assume the ground truth $\sigma(x)$ is
  an indicator function supported at the two yellow squares at the top
  left and bottom right corners; see Figure \ref{fig:media}.}
\begin{figure}[htp]
	\centering
	\subfloat[Ground Truth\label{fig:mediaGT}]{\includegraphics[width=0.45\textwidth]{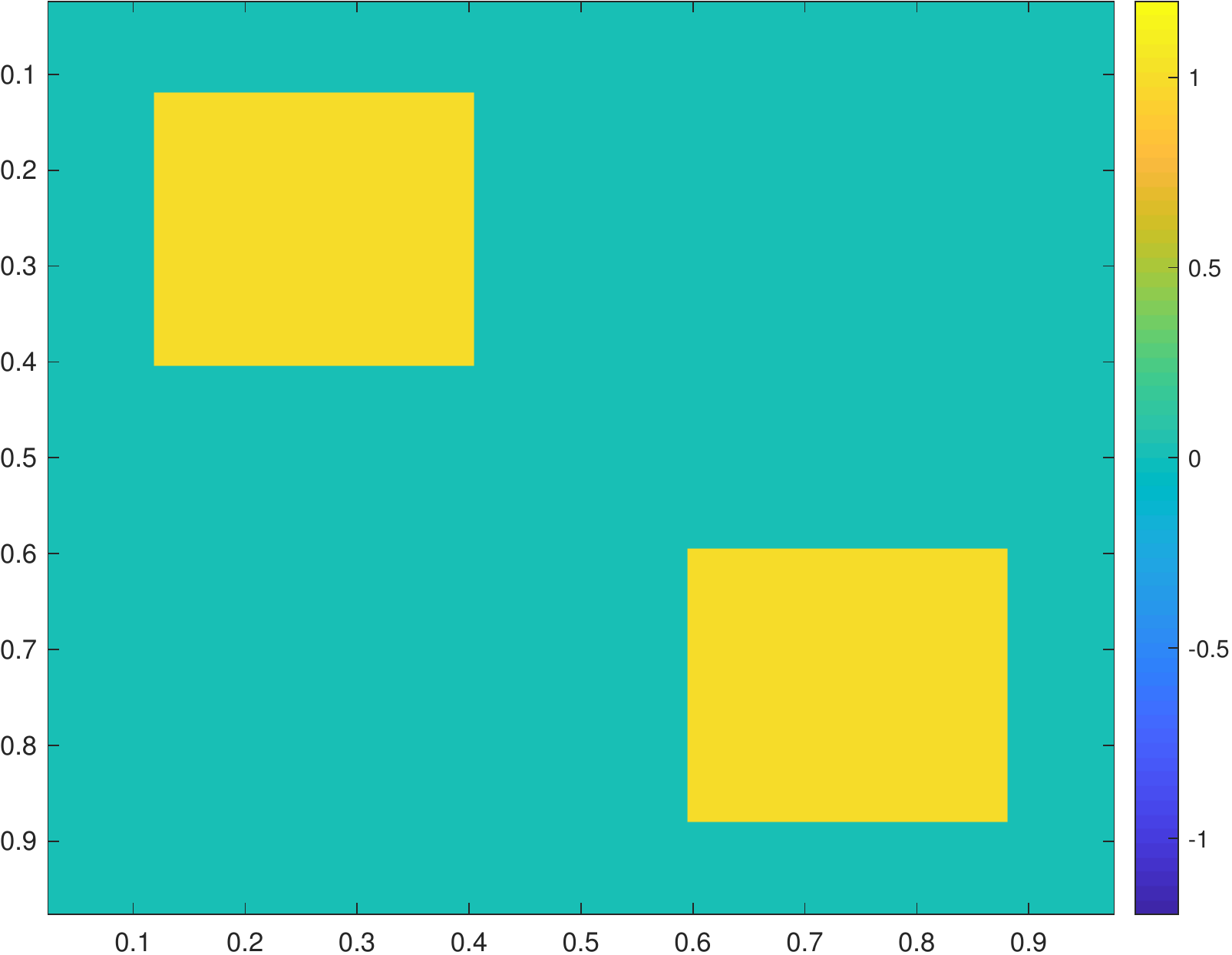}}\hfill
	\subfloat[Case 1\label{fig:media1}]{\includegraphics[width=0.45\textwidth]{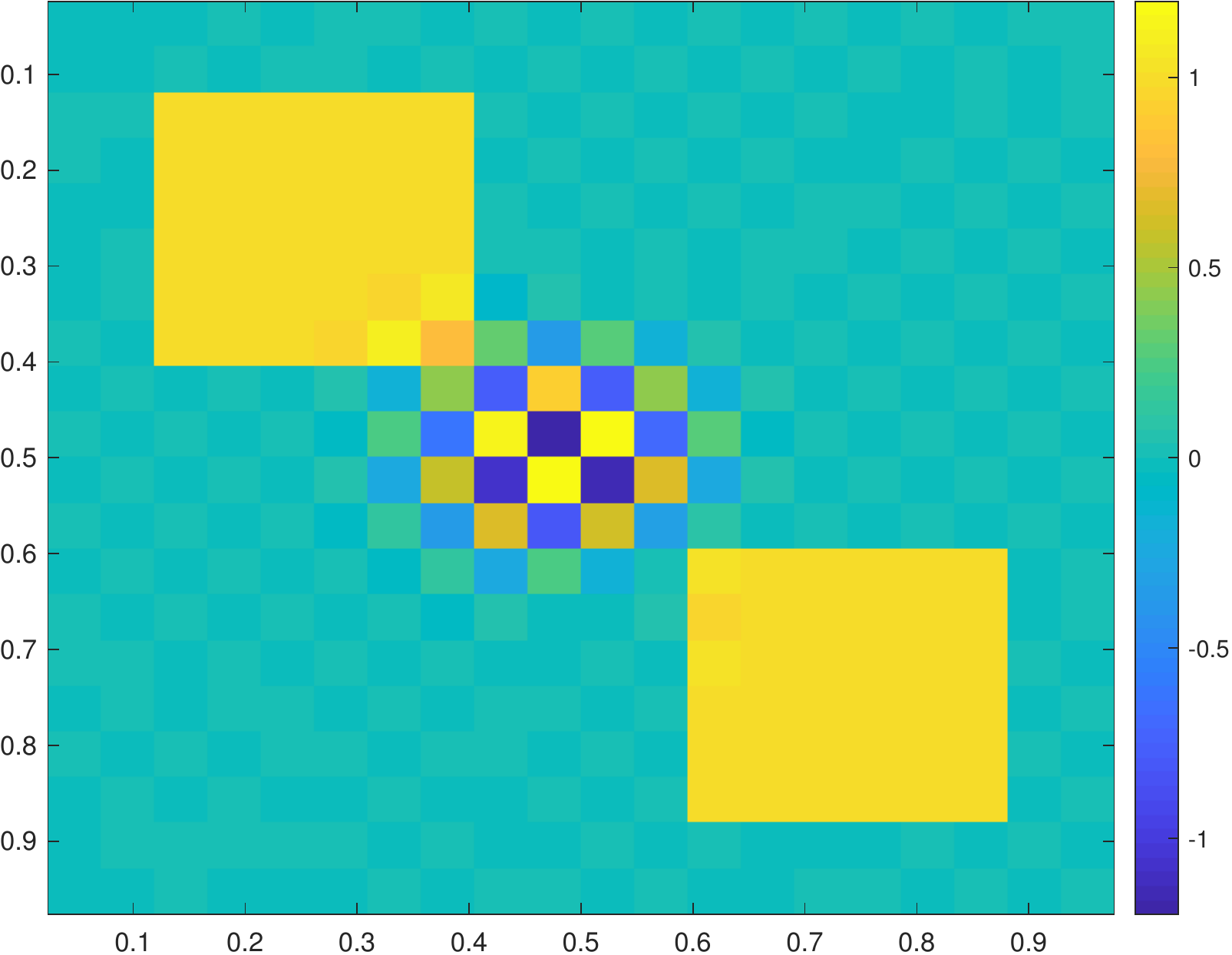}}\hfill
	\subfloat[Case 2\label{fig:media2}]{\includegraphics[width=0.45\textwidth]{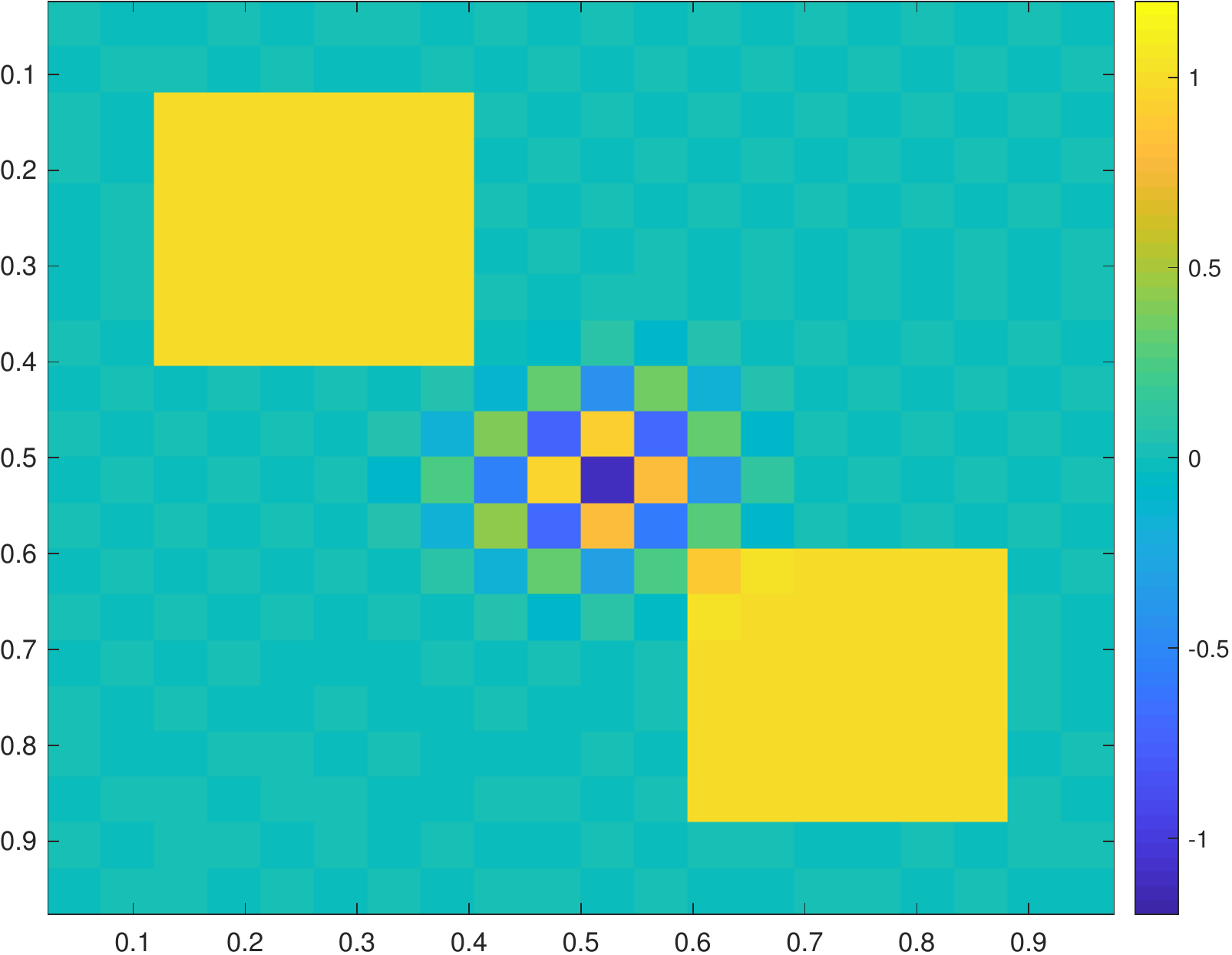}}\hfill
	\subfloat[Gaussian\label{fig:mediaGauss}]{\includegraphics[width=0.45\textwidth]{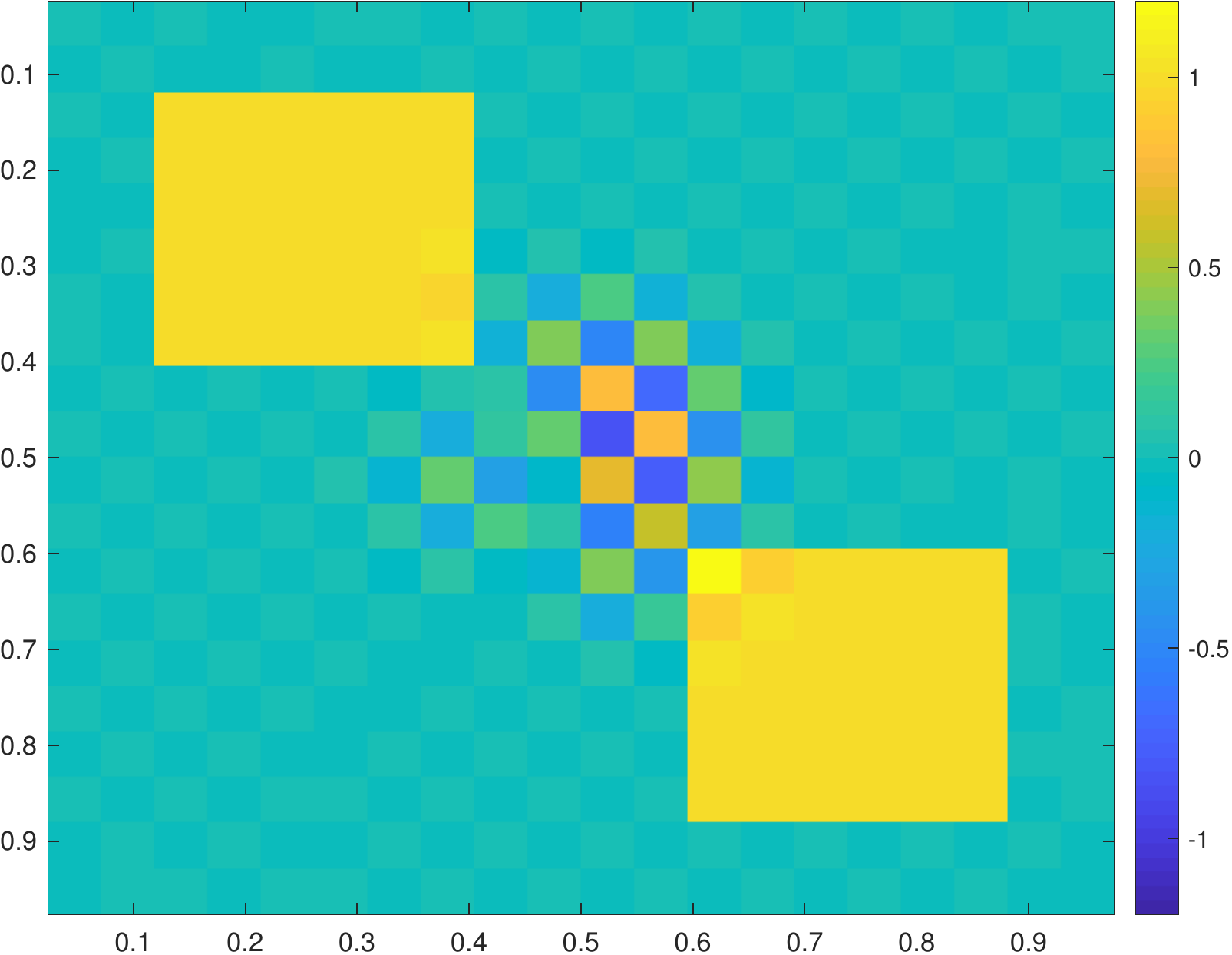}}\hfill
	\caption{The ground truth media and the reconstructed media via all three sketching strategies.}
	\label{fig:media}
\end{figure}
\blue{The background media $\sigma^\ast(x)$
  (cf. \eqref{eqn:background}) is set to be a constant function with
  value $10$. We use finite element method to calculate $\rho_1(x)$
  and $\rho_2(x)$ on a uniform mesh with $\Delta x=1/20$. The
  associated boundary conditions $\phi$ and $\psi$ are constructed as
  Dirac-delta functions at all boundary grid points. Under this setup,
  the matrix $\Asf$ has dimensions $10^4 \times 400$. The right-hand
  side $\bsf$ is generated by multiplying $\Asf$ with the ground truth
  $\sigma(x)$ and adding white noise. The EIT inverse problem is
  highly ill-posed, and thus we set the standard deviation of the mean zero Gaussian noise
  to be small: $10^{-8}$.  All three strategies are tested with
  different number of rows. We record the relative error
  \eqref{eqn:relerrResidue} by taking $10$ independent trials.}

\blue{In Figure~\ref{fig:media}, we plot the ground truth media
  $\sigma(x)$ and the reconstructed media using all three different
  strategies, with $r=74^2=5476$. All of them can roughly reconstruct
  the unknown function with some oscillatory errors in the center of
  the domain. In Figure~\ref{fig:EIT}, we plot the relative error in
  terms of the number of rows $r$ in the sketching matrix $\Ssf$ ($r$
  is set to be $26^2$, $38^2$, $50^2$, $62^2$, and $74^2$). We see
  that the Case-2 strategy performs as well as the unstructured
  Gaussian reference, and they both outperform Case 1.}
\begin{figure}[htp]
	\centering
	\includegraphics[width=0.6\textwidth]{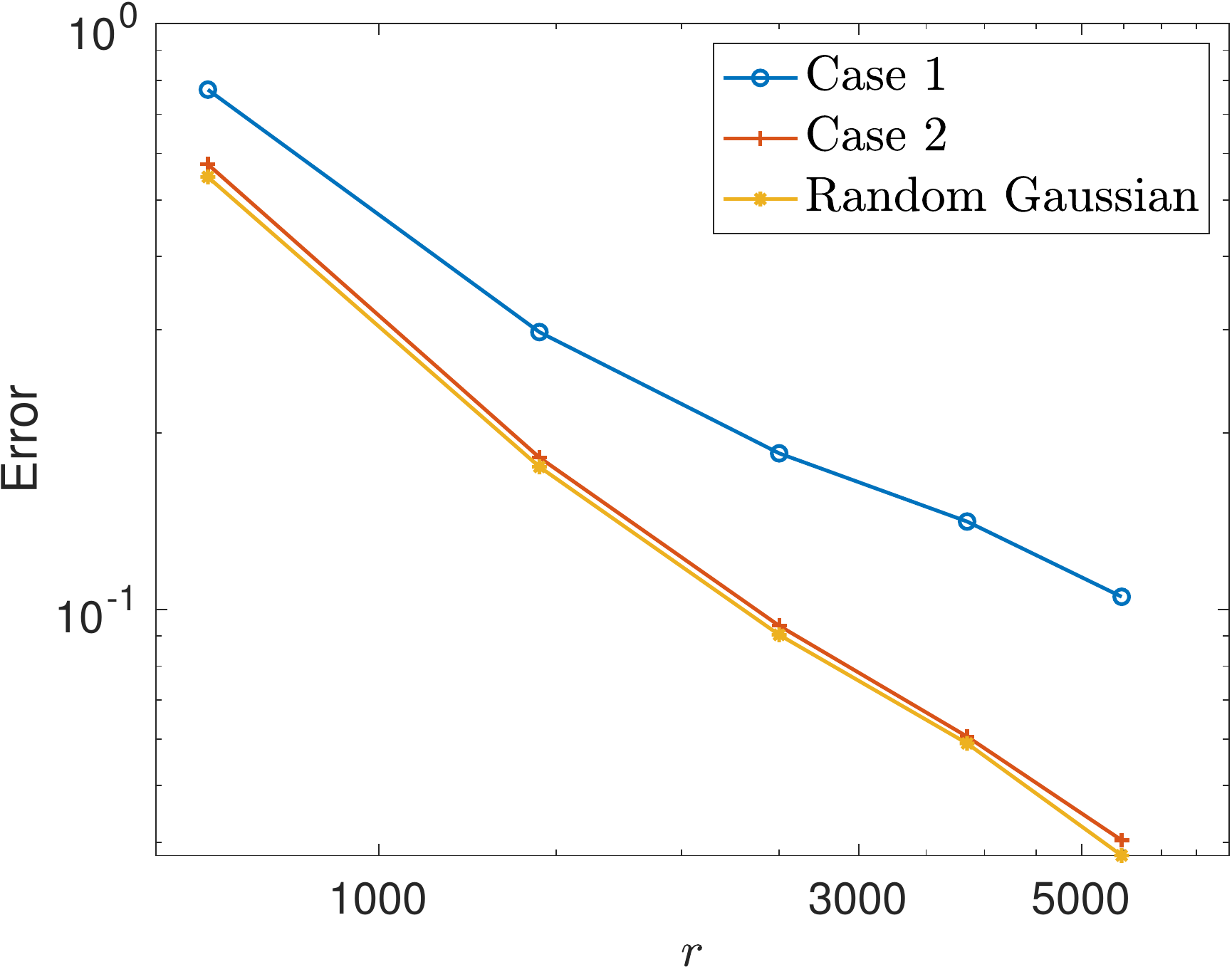}
	\caption{For all three strategies, the relative error
          decreases as the number of rows in $\Ssf$ increases. In
          particular, the Case-2 sketching strategy performs as well
          as the unstructured Gaussian strategy.}
	\label{fig:EIT}
\end{figure}

\section{Concluding remarks}
\blue{Most PDE-based inverse problems, upon linearization, become
  Fredholm integral equations, with the testing functions being the
  product of two functions that are solutions to the forward and the
  adjoint PDEs. A Khatri-Rao matrix structure arises in the
  discretization. We study the sketching problem for matrices of this
  type, where a corresponding structure is enforced in the sketching
  matrix, for efficiency of computation.  We construct the problem
  under the $(\epsilon,\delta)$-$l^2$ embedding framework, and
  investigate the number of rows of the sketching matrix that are
  needed to reconstruct the least-squares solution with $\epsilon$
  accuracy and $\delta$ confidence. The lower bounds differ for the
  two different sketching strategies that we propose, but both are
  independent of the size of the ambient space.}

\section*{Acknowledgments}
Chen, Li, and Newton are supported in part by NSF-DMS-1750488 and
NSF-TRIPODS 1740707. Wright is supported in part by NSF Awards 1628384, 1634597, and 1740707;  Subcontract 8F-30039 from Argonne National Laboratory; and Award N660011824020 from the DARPA Lagrange Program.

\appendix 
% \begin{appendix}
  \section{Key Identities and Inequalities}\label{sec:app1}
  
Some identities and inequalities used repeatedly in the text are
collected here.

\subsection{Identities of the Kronecker product}
Let $\mathsf{A}=(a_{ij})\in\Rbb^{r_1\times n_1}$,
$\mathsf{B}=(b_{ij})\in\Rbb^{r_1\times n_2}$. Then the Kronecker
product of $\Asf$ and $\Bsf$ forms a matrix of size $r_1r_2\times
n_1n_2$ defined by:
\[
\Amat\otimes\Bmat = \begin{bmatrix}
a_{11}\Bmat & a_{12}\Bmat & \cdots & a_{1n_1}\Bmat\\
a_{21}\Bmat & \cdots & \cdots & a_{2n_1}\Bmat\\
\ldots & \ddots & \ddots & \ldots\\
a_{r_11}\Bmat & a_{r_12}\Bmat & \cdots & a_{r_1n_1}\Bmat
\end{bmatrix}\,.
\]
The following properties hold.
\begin{enumerate}
\item Let $\mathsf{A}\in\Rbb^{r_1\times n_1}$,
  $\mathsf{B}\in\Rbb^{r_2\times n_2}$, $\mathsf{C}\in\Rbb^{n_1\times
  p_1}$ and $\mathsf{D}\in\Rbb^{n_2\times p_2}$, then we have the
  mixed-product property:
\begin{equation}\label{eqn:kron1}
(\mathsf{A}\otimes \mathsf{B})(\mathsf{C}\otimes \mathsf{D}) = (\mathsf{A}\mathsf{C})\otimes  (\mathsf{B}\mathsf{D})\,.
\end{equation}
%		\item \label{itm:kron2} Define \textit{row-wise vectorization} as $\Xsf \mapsto \textbf{vec}(\Xsf)$ and the \textit{flip} operator $F_{c,d}: \Rbb^{cd} \rightarrow \Rbb^{cd}$ as the linear extension of the mapping $\esf_i\otimes \esf_j \mapsto \esf_j\otimes \esf_i$, then for any $\Asf\in \Rbb^{a\times b }, \mathsf{B}\in\Rbb^{c\times d}$ and $\Xsf \in \Rbb^{b\times d}$, we have
%		\[
%		(\Asf \otimes \mathsf{B}) \textbf{vec}(\Xsf) = F_{b,a} (\mathsf{B}\otimes \Asf) F_{c,d} \textbf{vec}(\Xsf)
%		\]
%		and that
%		\[
%		F_{c,d}\textbf{vec}(\Xsf) = \textbf{vec}(\Xsf^\top)
%		\]
%		\item \label{itm:kron3} The flip operator $F_{c,d}$ is unitary.
\item Let $\mathsf{A}\in\Rbb^{r_1\times n_1}$,
  $\mathsf{B}\in\Rbb^{r_2\times n_2}$, and $\Xsf\in\Rbb^{n_1\times
  n_2}$. Further denote by $\textbf{vec}(\Xsf)$ the vectorization of
  $\Xsf$ formed by stacking the columns of $\Xsf$ into a single column
  vector, then
\begin{equation}\label{eqn:kron2}
(\mathsf{B} \otimes \mathsf{A}) \textbf{vec}(\Xsf) = \textbf{vec}(\mathsf{A} \Xsf \mathsf{B}^\top)\,.
\end{equation} 
Equivalently, given the same $\Asf,\mathsf{B}$ and $\xsf \in \mathbb R^{n_1n_2}$, denote $\textbf{Mat}(\xsf) \in \Rbb^{n_1\times n_2}$ the matricization of the vector $\xsf$ by aligning subvectors of $\xsf$ that are of length $n_1$ into a matrix with $n_2$ columns, then
\begin{equation}\label{eqn:kron3}
	(\mathsf{B} \otimes \mathsf{A})\xsf = \mathbf{vec}\left(\mathsf{A} \textbf{Mat}(\xsf)\mathsf{B}^\top \right)\,.
\end{equation}

\end{enumerate}

\subsection{Sub-exponential random variables and Bernstein inequality}\label{sec:app2}

Properties of sub-exponential random variables used in the proofs are
defined here.

\begin{definition}{\textbf{Sub-Exponential random variable}}
A random variable $X\in \Rbb$ is said to be sub-exponential with
parameters $(\lambda,b)$ (denoted as $X\sim \text{subE}(\lambda,b)$)
if $\Ebb X = 0$ and its moment generating function satisfies
\begin{equation}\label{eqn:sub_exp}
\Ebb e^{sX} \leq \exp\left(\frac{s^2\lambda^2}{2}\right)\,,\quad \mbox{for all} \;\; |s| \leq \frac{1}{b}\,.
\end{equation}
%% \sw{OK, I found this definition in a draft of Wainwright (2015) that I
%%   put in the references. But should we give a citation? Other sources
%%   I found on the web gave slightly different definitions.} \kit{we could cite another Vershynin paper I found on arXiv, it has the MGF information essentially the same, I think. I'll put it in the references on dropbox.}
\end{definition}

We have the following.
\begin{proposition}\label{prop:gaussian_square}
Let $Z\sim \mathcal{N}(0,1)$, then $X\defeq Z^2 -1 $ is
sub-exponential with parameters $(2,4)$.
\end{proposition}

We conclude with the well known Bernstein inequality.
\begin{proposition}[Bernstein inequality]
Let $X_1,\ldots,X_n$ be i.i.d. mean zero random variables. Suppose that $|X_i|\leq M$ for all $i=1,\ldots,n$, then for any $t>0$,
\begin{equation}\label{eqn:Bernstein}
\Pr\left(\sum_{i=1}^n X_i \geq t \right) \leq \exp\left(-\frac{t^2/2}{\sum_{i=1}^n \Ebb\left[X_i^2\right] + Mt/3 }\right)\,.
\end{equation}
\end{proposition}

\bibliographystyle{amsplain}
\bibliography{ref}
	
\end{document}